\documentclass[a4paper,11pt]{article}
\usepackage[all,ps,arc]{xy}  
\usepackage{amsmath,amssymb,amsthm,latexsym}
\usepackage{pdfsync}
\usepackage{makeidx}
\usepackage[english]{babel}
\usepackage[applemac]{inputenc}
\usepackage{graphicx}
\usepackage[x11names]{xcolor}
\usepackage{geometry}
\newtheorem{Theorem}{Theorem}[section]
\newtheorem{Lemma}[Theorem]{Lemma}
\newtheorem{Proposition}[Theorem]{Proposition}
\newtheorem{Definition}[Theorem]{Definition}
\newtheorem{Corollary}[Theorem]{Corollary}
\newtheorem{Remark}[Theorem]{Remark}
\newtheorem{Example}[Theorem]{Example}
\newtheorem{Notation}[Theorem]{Notation}

\newtheorem{Text}[Theorem]{}

\newcommand{\cA}{{\mathcal A}}
\newcommand{\cB}{{\mathcal B}}
\newcommand{\cC}{{\mathcal C}}
\newcommand{\cD}{{\mathcal D}}
\newcommand{\cE}{{\mathcal E}}
\newcommand{\cF}{{\mathcal F}}

\newcommand{\cH}{{\mathcal H}}
\newcommand{\cI}{{\mathcal I}}
\newcommand{\cId}{{\mathcal Id}}

\newcommand{\cM}{{\mathcal M}}
\newcommand{\cN}{{\mathcal N}}
\newcommand{\cP}{{\mathcal P}}
\newcommand{\cQ}{{\mathcal Q}}
\newcommand{\cR}{{\mathcal R}}
\newcommand{\cS}{{\mathcal S}}
\newcommand{\cT}{{\mathcal T}}
\newcommand{\cU}{{\mathcal U}}
\newcommand{\cZ}{{\mathcal Z}}

\newcommand\Arr{\mathbf{Arr}}

\newcommand\Grp{\mathbf{Grp}}
\newcommand\XMod{\mathbf{XMod}}
\newcommand\Null{\mathbf{Null}}

\newcommand\Ker{\mathrm{Ker}}

\newcommand\id{\mathrm{id}}
\newcommand\Id{\mathrm{Id}}
\newcommand\ClId{\mathrm{ClId}}
\newcommand\Idl{\mathrm{Idl}}
\newcommand\Retr{\mathrm{Retr}}

\newcommand\ob{\mathrm{ob}}
\newcommand\arr{\mathrm{ar}}
\newcommand\conj{\mathrm{conj}}

\begin{document}

\title{Homotopy torsion theories}

\author{Sandra Mantovani\footnote{Dipartimento di matematica, Universit\`a degli studi di Milano,
Via C. Saldini 50, 20133 Milano, Italia, sandra.mantovani@unimi.it} , 
Mariano Messora\footnote{Dipartimento di matematica, Universit\`a degli studi di Milano,
Via C. Saldini 50, 20133 Milano, Italia, mariano.messora@unimi.it} ,
Enrico M. Vitale\footnote{Institut de recherche en math\'ematique et physique, Universit\'e catholique de Louvain,
Chemin du Cyclotron 2, B 1348 Louvain-la-Neuve, Belgique, enrico.vitale@uclouvain.be}} 

\maketitle

\noindent {\bf Abstract:}
\noindent In the context of categories equipped with a structure of nullhomotopies, we introduce the notion of homotopy torsion theory.
As special cases, we recover pretorsion theories as well as torsion theories in multi-pointed categories and in pre-pointed categories.
Using the structure of nullhomotopies induced by the canonical string of adjunctions between a category $\cA$ and the category
$\Arr(\cA)$ of arrows, we give a new proof of the correspondence between orthogonal factorization systems in $\cA$ and homotopy 
torsion theories in $\Arr(\cA),$ avoiding the request on the existence of pullbacks and pushouts in $\cA.$ Moreover, such a correspondence 
is extended to weakly orthogonal factorization systems and weak homotopy torsion theories.

\tableofcontents

\section{Introduction}\label{SecIntro}

There is a striking analogy between orthogonal factorization systems and torsion theories. If you have an orthogonal factorization
system in a category, from each arrow you get a pair of composable arrows, and the two arrows lie in two assigned classes of arrows 
whose intersection is reduced to the class of isomorphisms. If you have a torsion theory in an abelian category, from each object you 
get a pair of objects connected by a short exact sequence, and the two objects lie in two assigned subcategories whose intersection is 
reduced to the zero object. This analogy becomes even more strict considering notions of torsion theory adapted to general (not 
necessarily abelian) categories.

Since the category $\Arr(\cA)$ of arrows of a category $\cA$ is the standard way to turn arrows into objects, one expects that a precise relation 
between orthogonal factorization systems in $\cA$ and torsion theories in $\Arr(\cA)$ exists. 
A new direction to establish and describe such a relation has been indicated in the unpublished talk \cite{Sandra}. 
More recently, and independently from \cite{Sandra}, the correspondence between orthogonal factorization systems in $\cA$ and a certain kind of torsion 
theories in $\Arr(\cA)$ has been established in \cite{GJ}. The proof of the main result in \cite{GJ} depends on previous results on (co)monads
from \cite{GT} and on the reformulation of the definition of orthogonal factorization system given in \cite{KT}.

The notion of torsion theory used in \cite{GJ} is based on the notion of pre-pointed category and, more precisely, on pre-(co)kernels 
(not to be confused with the prekernels studied in \cite{FF, FFG}: pre-kernels from \cite{GJ} are weak prekernels). 
Pre-kernels are not defined in terms of a universal property, they are defined via an ad-hoc construction which depends on the 
reflective and coreflective character of the subcategory of trivial objects. Nevertheless, pre-kernels have a universal property: 
they are a special instance of the so-called homotopy kernels, as pointed out in \cite{SnailEV} in the special case of $\Arr(\cA).$

Any adjunction (in fact, any pre-radical and any pre-coradical) induces a structure of nullhomotopies (see Definition \ref{DefNullHom}), 
and any structure of nullhomotopies 
carries with it a notion of homotopy kernel. Therefore, the idea developed in the present paper will allow us to gather in the same framework
various notions of pretorsion and 
torsion theory (in pointed categories, in multi-pointed categories, in pre-pointed categories) appearing in the literature.
We will show that the latter are all special cases of a general notion based on homotopy kernels.
We call such a notion homotopy torsion theory and we formulate its definition in any category 
equipped with a structure of nullhomotopies. The main test for our definition is to get a revisited version of the result in \cite{GJ} cited above: 
we give a self-contained proof that orthogonal factorization systems in a category $\cA$ correspond to homotopy torsion theories in $\Arr(\cA)$ 
with respect to the nullhomotopy structure induced by the canonical string of adjunctions between $\cA$ and $\Arr(\cA).$
Moreover, basically the same proof allows us to extend this result to a correspondence between weakly orthogonal factorization systems
and weak homotopy torsion theories.

The layout of the paper is as follows.
In order to express the notion of homotopy torsion theory, in Section \ref{SecNullRad} we recall the notion of category with nullhomotopies. 
We give a method to construct examples from pre-(co)radicals and we compare structure of nullhomotopies with ideals of arrows. 
The case of pre-pointed categories and, in particular, the example of $\Arr(\cA),$ are considered in Section \ref{SecPrePArr}.
In Section \ref{SecHKHC} and in Section \ref{SecExistHK} we discuss the standard notion of homotopy kernel: definition, 
main properties and conditions for the existence. In Section \ref{SecCharPrep} we characterize pre-pointed categories among categories equipped 
with a structure of nullhomotopies. In Section \ref{SecHTT} we define homotopy torsion theories with respect to a given structure of nullhomotopies
and, in Section \ref{SecPretorsion}, we compare them with pretorsion theories. It turns out that pretorsion theories coincide with homotopy torsion 
theories when the structure of nullhomotopies is reduced to a closed ideal of arrows. 
The correspondence 
between various types of factorization systems in a category $\cA$ and homotopy torsion theories in $\Arr(\cA)$ is developed in the rest of 
the paper: the case of (weakly) orthogonal factorization systems in Section \ref{SecFSHTT}, the case of quasi-proper and proper orthogonal factorization
systems in Section \ref{SecQP} and in Section \ref{SecProper}, and finally the pointed case in Section \ref{SecPointed}. 
In order to help the reader, Section \ref{SecPanoramic} gives a global view of the correspondences studied in Sections \ref{SecFSHTT},
\ref{SecQP} and \ref{SecProper}.

Finally, let us mention that large portions of this paper also appear in the Master Thesis \cite{MM22} of the second author, written under the 
supervision of the first author. 

\section{Nullhomotopies, ideals and pre-radicals}\label{SecNullRad}

For the notion of structure of nullhomotopies on a category, we follow \cite{SnailEV,JMMVFibr} and adopt a definition a bit stronger than the 
original one in \cite{GR97}.

\begin{Definition}\label{DefNullHom}{\rm
A {\it structure of nullhomotopies} $\Theta$ on a category $\cB$ is given by:
\begin{enumerate}
\item[1)] For every arrow $g$ in $\cB,$ a set $\Theta(g)$ whose elements are called nullhomotopies on $g.$
\item[2)] For every triple of composable arrows $\xymatrix{W \ar[r]^f & X \ar[r]^g & Y \ar[r]^h & Z},$ a map 
$$h \circ - \circ f \colon \Theta(g) \to \Theta(h \cdot g \cdot f)$$
such that, for every $\varphi \in \Theta(g),$ one has
\begin{enumerate}
\item $(h' \cdot h) \circ \varphi \circ (f \cdot f') = h' \circ (h \circ \varphi \circ f) \circ f'$ whenever the compositions $h' \cdot h$ and $f \cdot f'$ 
are defined,
\item $\id_Y \circ \varphi \circ \id_X = \varphi.$
\end{enumerate}
\end{enumerate}
}\end{Definition}

\begin{Notation}\label{NotNull}{\rm
To visualize a nullhomotopy $\lambda \in \Theta(g)$ in a diagram, we will sometimes use the notation
$$\xymatrix{X \ar@/^1.0pc/[rr]^{g} \ar@/_1.0pc/@{-->}[rr]_{} \ar@{}[rr]|{\lambda \; \Uparrow} & & Y}$$
This notation comes from the fact that, in a 2-category with a zero object, we get a structure of nullhomotopies by taking as 
nullhomotopies on an arrow $g$ all the 2-cells from the zero arrow to $g$ (or from $g$ to the zero arrow).
}\end{Notation}

\begin{Remark}\label{RemDefAlter}{\rm
When, in Definition \ref{DefNullHom}, $f=\id_X$ or $h=\id_Y,$ we write $h \circ \varphi$ and $\varphi \circ f$ instead of $h \circ \varphi \circ \id_X$ 
and $\id_Y \circ \varphi \circ f.$ It is possible to restate Definition \ref{DefNullHom} using $h \circ \varphi$ and $\varphi \circ f$ as primitive 
operations and asking that $h \circ (\varphi \circ f) = (h \circ \varphi) \circ f, h' \circ (h \circ \varphi) = (h' \cdot h) \circ \varphi, 
(\varphi \circ f) \circ f' = \varphi \circ (f \cdot f'), \id_Y \circ \varphi = \varphi = \varphi \circ \id_X.$
}\end{Remark}

There is an obvious notion of morphism between structures of nullhomotopies. It is a special case of the notion of morphism between 
categories with nullhomotopies from \cite{EVhcc}. (If $\cB$ is a category, we denote by $\arr(\cB)$ the possibly large set of its arrows, 
and by $\ob(\cB)$ the possibly large set of its objects.)

\begin{Definition}\label{DefMorphNull}{\rm
Let $\Theta$ and $\Theta'$ be two structures of nullhomotopies on the same category $\cB.$ A {\it morphism} $\cI \colon \Theta \to \Theta'$ 
is given by a family of maps indexed by the arrows of $\cB$
$$\{ \cI_g \colon \Theta(g) \to \Theta'(g) \}_{g \in \arr(\cB)}$$
such that, for every triple of composable arrows 
$\xymatrix{W \ar[r]^f & X \ar[r]^g & Y \ar[r]^h & Z},$ the following diagram commutes
$$\xymatrix{\Theta(g) \ar[rr]^{\cI_g} \ar[d]_{h \circ - \circ f} & & \Theta'(g) \ar[d]^{h \circ - \circ f} \\
\Theta(h \cdot g \cdot f) \ar[rr]_{\cI_{h \cdot g \cdot f}} & & \Theta'(h \cdot g \cdot f)}$$
The structures of nullhomotopies on $\cB$ together with their morphisms constitute a category denoted by $\Null(\cB).$
}\end{Definition}

Pretorsion theories, considered in Section \ref{SecPretorsion}, are based on the notion of ideal of arrows, which provides (via Lemma \ref{LemmaNullId}) our 
first example of structure of nullhomotopies. We start recalling from \cite{GJ} the notion of (closed) ideal of arrows together with some basic facts about it.

\begin{Definition}\label{DefIdeal}{\rm
Let $\cB$ be a category.
\begin{enumerate}
\item A subset $\cZ_1 \subseteq \arr(\cB)$ is an {\it ideal} if it satisfies the following condition: if $g \colon X \to Y$ is in $\cZ_1,$ then for 
any pair of arrows $f \colon W \to X$ and $h \colon Y \to Z,$ the composite arrow $h \cdot g \cdot f \colon W \to Z$ is still in $\cZ_1.$
\item If $\cZ_1$ is an ideal, an object $N$ is $\cZ_1$-{\it trivial} if $\id_N \colon N \to N$ is in $\cZ_1.$
\item An ideal $\cZ_1$ is {\it closed} when an arrow $g$ is in $\cZ_1$ (if and) only if it factors through some $\cZ_1$-trivial objects.
\end{enumerate}
}\end{Definition}

\begin{Text}\label{TextRetrId}{\rm 
Consider the following constructions ($\cP(X)$ denotes the poset of subsets of a set $X$):
\begin{enumerate}
\item[-] $i \colon \cP(\ob(\cB)) \to \cP(\arr(\cB)), \; i(\cZ_0) = \{ g \in \arr(\cB) \mid g \mbox{ factors through some objects in } \cZ_0 \}$
\item[-] $t \colon \cP(\arr(\cB)) \to \cP(\ob(\cB)), \; t(\cZ_1) =  \{ \cZ_1\mbox{-trivial objects} \} = \{ N \in \ob(\cB) \mid \id_N \in \cZ_1 \}$
\end{enumerate}
Plainly, we have:
\begin{enumerate}
\item If $\cZ_0 \subseteq \cZ_0',$ then $i(\cZ_0) \subseteq i(\cZ_0').$ If $\cZ_1 \subseteq \cZ_1',$ then $t(\cZ_1) \subseteq t(\cZ_1').$
\item For any $\cZ_0 \subseteq \ob(\cB),$ $i(\cZ_0)$ is a closed ideal and $\cZ_0 \subseteq t(i(\cZ_0)).$
\item If $\cZ_1$ is an ideal, then $t(\cZ_1)$ is closed under retracts and $i(t(\cZ_1)) \subseteq \cZ_1.$
\item The set $t(i(\cZ_0))$ of $i(\cZ_0)$-trivial objects is the smallest subset of $\ob(\cB)$ containing $\cZ_0$ and closed under retracts. 
\item If $\cZ_1$ is an ideal, the set $i(t(\cZ_1))$ is the largest closed ideal contained in $\cZ_1.$
\end{enumerate}
Moreover, the constructions $i \colon \cP(\ob(\cB)) \to \cP(\arr(\cB))$ and $t \colon \cP(\arr(\cB)) \to \cP(\ob(\cB))$ restrict to an isomorphism 
$$\Retr(\cB) \simeq \ClId(\cB)$$
where $\Retr(\cB) \subseteq \cP(\ob(\cB))$ is the poset of the subsets of $\ob(\cB)$ closed under retracts, and $\ClId(\cB) \subseteq \cP(\arr(\cB))$ 
is the poset of closed ideals.
}\end{Text}

\begin{Text}\label{TextDiscStr}{\rm
Now we compare ideals and structures of nullhomotopies: structures of nullhomotopies are a wide generalization of ideals which correspond 
to discrete structures. We call a structure of nullhomotopies $\Theta$ {\it discrete} when, for every arrow $g,$ the set $\Theta(g)$ is either the 
singleton or the empty set. }\end{Text}

\begin{Lemma}\label{LemmaNullId}
Let $\cB$ be a category and $\Idl(\cB)$ the poset of ideals of arrows in $\cB.$ Then $\Idl(\cB)$ is equivalent to the reflective subcategory of 
$\Null(\cB)$ spanned by the discrete structures.
\end{Lemma}

\begin{proof}
The full and faithful functor $\Idl(\cB) \to \Null(\cB)$ is realized by associating with an ideal $\cZ_1$ the discrete structure on $\cZ_1 \colon$
$$\Theta_{\cZ_1}(g) = \left\{ \begin{array}{rcl} \ast & \mbox{if} & g \in \cZ_1 \\ \emptyset & \mbox{if} & g \notin \cZ_1 \end{array} \right.$$
The reflection $\Null(\cB) \to \Idl(\cB)$ associates with a structure of nullhomotopies $\Theta$ the ideal $\cZ_1(\Theta)$ of the arrows 
$g \in \arr(\cB)$ such that $\Theta(g)$ is non-empty.
\end{proof}

\begin{Example}\label{ExNullSubcat}{\rm
Here there is another easy way to construct structures of nullhomotopies.
Let $\cU \colon \cA \to \cB$ be a full and faithful functor. We get a structure of nullhomotopies $\Theta_{\cA}$ on $\cB$ by putting, 
for arrows  $\xymatrix{W \ar[r]^-{f} & X \ar[r]^-{g} & Y \ar[r]^-{h} & Z },$
$$\Theta_{\cA}(g) = \{ (g_1,A,g_2) \mid g = g_2 \cdot g_1 \colon X \to \cU A \to Y, A \in \cA \}, \;
h \circ (g_1,A,g_2) \circ f = (g_1 \cdot f, A, h \cdot g_2)$$
}\end{Example} 

\begin{Text}\label{TextRad}{\rm
Usually, a pre-radical in a category $\cB$ is defined as a subfunctor $\cR$ of the identity functor $\cId \colon \cB \to \cB.$ We extend the 
terminology to any natural transformation $\beta \colon \cR \Rightarrow \cId$ on any endofunctor $\cR \colon \cB \to \cB.$ We use 
pre-radicals and pre-coradicals in $\cB$ to construct structures of nullhomotopies.
\begin{enumerate}
\item If $\beta \colon \cR \Rightarrow \cId \colon \cB \to \cB$ is a pre-radical, we get a structure of nullhomotopies on $\cB$ by putting
$\Theta_{\beta}(g) = \{ \psi \colon X \to \cR Y \mid \beta_Y \cdot \psi = g \}$ and $h \circ \psi \circ f = \cR(h) \cdot \psi \cdot f$
$$\xymatrix{ & \cR Y \ar[rd]^{\beta_Y} \\ X \ar[ru]^{\psi} \ar[rr]_{g} & & Y }$$
\item If $\gamma \colon \cId \Rightarrow \cS \colon \cB \to \cB$ is a pre-coradical, we get a structure of nullhomotopies on $\cB$ by putting
$\Theta_{\gamma}(g) = \{ \varphi \colon \cS X \to Y \mid \varphi \cdot \gamma_X = g \}$ and $h \circ \varphi \circ f = h \cdot \varphi \cdot \cS(f)$
$$\xymatrix{ & \cS X \ar[rd]^{\varphi} \\ X \ar[ru]^{\gamma_X} \ar[rr]_{g} & & Y }$$
\item If $\beta \colon \cR \Rightarrow \cId$ is a pre-radical and $\gamma \colon \cId \Rightarrow \cS$ is a pre-coradical, we get a structure of 
nullhomotopies on $\cB$ by putting $\Theta_{\gamma,\beta}(g) = \{ \lambda \colon \cS X \to \cR Y \mid \beta_Y \cdot \lambda \cdot \gamma_X = g \}$ 
and $h \circ \lambda \circ f = \cR(h) \cdot \lambda \cdot \cS(f)$
$$\xymatrix{ & \cS X \ar[r]^{\lambda} & \cR Y \ar[rd]^{\beta_Y} \\ X \ar[ru]^{\gamma_X} \ar[rrr]_{g} & & & Y }$$
\end{enumerate}
}\end{Text}

\begin{Remark}\label{RemTriv}{\rm
Two comments on the constructions in \ref{TextRad}.
\begin{enumerate}
\item Observe that, if $\beta$ in \ref{TextRad}.1 is an isomorphism, then the structure $\Theta_{\beta}$ is the terminal object of $\Null(\cB).$
This means that, for any arrow $g$ in $\cB,$ the set $\Theta_{\beta}(g)$ is reduced to a singleton. Indeed, the condition $\beta_Y \cdot \psi = g$ 
is equivalent to $\psi = \beta_Y^{-1} \cdot g.$ The same  happens if $\gamma$ in \ref{TextRad}.2 is an isomorphism and if both $\beta$ and 
$\gamma$ in \ref{TextRad}.3 are isomorphisms.
\item More in general, the structure $\Theta_{\beta}$ of \ref{TextRad}.1 is discrete if and only if, for every $Y \in \cB,$ the arrow $\beta_Y$ is a 
monomorphism.
The structure $\Theta_{\gamma}$ of \ref{TextRad}.2 is discrete if and only if, for every $X \in \cB,$ the arrow $\gamma_X$ is an epimorphism. 
\end{enumerate} 
}\end{Remark}

\section{Pre-pointed categories and the category of arrows}\label{SecPrePArr}

A string of adjunctions like the one involved in the next proposition is called a {\it pre-pointed category} in \cite{GJ}.

\begin{Proposition}\label{PropCUD}
Consider the following string of adjunctions
$$\xymatrix{ \cA \ar[rr]|{\cU} & & \cB \ar@<-1.5ex>[ll]_{\cC} \ar@<1.5ex>[ll]^{\cD} }
\;\;\;\;\; \cC \dashv \cU \dashv \cD$$
with units and counits given by 
$$\gamma_B \colon B \to \cU\cC B, \; \delta_A \colon \cC\cU A \to A, \;\; \alpha_A \colon A \to \cD\cU A,\;  \beta_B \colon \cU\cD B \to B.$$
If $\cU \colon \cA \to \cB$ is full and faithful, then:
\begin{enumerate}
\item The three structures of nullhomotopies on $\cA$ induced by the pre-radical $\delta$ and by the pre-coradical $\alpha$ are the terminal ones.
\item The three structures of nullhomotopies on $\cB$ induced by the pre-radical $\beta$ and by the pre-coradical $\gamma$ are isomorphic in $\Null(\cB).$
\end{enumerate}
\end{Proposition}

\begin{proof}
1. This follows from Remark \ref{RemTriv}.1 because $\delta$ and $\alpha$ are isomorphisms. \\
2. We check the isomorphism $\Theta_{\gamma,\beta} \simeq \Theta_{\gamma}.$ 
Fix an arrow $g \colon X \to Y$ in $\cB$ and define $\cI_g \colon \Theta_{\gamma,\beta}(g) \to \Theta_{\gamma}(g)$ by
$$\cI_g(\lambda \colon \cU\cC X \to \cU\cD Y) = (\beta_Y \cdot \lambda \colon \cU\cC X \to \cU\cD Y \to Y)$$
The fact that $\cI_g$ is well-defined is obvious and the fact that $\cI$ is a morphism of nullhomotopy structures comes from the naturality 
of $\beta.$ In the opposite direction, define $\cI_g^{-1} \colon \Theta_{\gamma}(g) \to \Theta_{\gamma,\beta}(g)$ by
$$\cI_g^{-1}(\varphi \colon \cU\cC X \to Y) = (\cU\cD(\varphi) \cdot \cU(\alpha_{\cC X}) \colon \cU\cC X \to \cU\cD\cU\cC X \to \cU\cD Y)$$
The fact that $\cI_g^{-1}$ is well-defined comes from the naturality of $\beta$ and the fact that $\cI^{-1}$ is a morphism of nullhomotopy 
structures comes from the naturality of $\alpha.$ The fact that $\cI_g(\cI_g^{-1}(\varphi)) = \varphi$ is attested by the commutativity of the 
following diagram, where the triangle on the left commutes by one of the triangular identities and the triangle on the right commutes by 
naturality of $\beta \colon$
$$\xymatrix{\cU\cC X \ar[rr]^-{\cU(\alpha_{\cC X})} \ar[rrd]_{\id} & & \cU\cD\cU\cC X \ar[rr]^-{\cU\cD(\varphi)} \ar[d]^{\beta_{\cU\cC X}} 
& & \cU\cD Y \ar[rr]^-{\beta_Y} & & Y \\
& & \cU\cC X \ar[rrrru]_{\varphi} }$$
The fact that $\cI_g^{-1}(\cI_g(\lambda)) = \lambda$ is attested by the commutativity of the following diagram:
$$\xymatrix{\cU\cC X \ar[rr]^-{\cU(\alpha_{\cC X})} \ar[rrd]_{\id} & & \cU\cD\cU\cC X \ar[rr]^-{\cU\cD(\lambda)} \ar[d]^{\beta_{\cU\cC X}} 
& & \cU\cD\cU\cD Y \ar[d]_{\beta_{\cU\cD Y}} \ar[rr]^-{\cU\cD(\beta_Y)} & & \cU\cD Y \\
& & \cU\cC X \ar[rr]_-{\lambda} & & \cU\cD Y \ar[rru]_{\id} }$$
The triangle on the left commutes by one of the triangular identities, the square commutes by naturality of $\beta.$ Finally, the triangle on the 
right commutes since, by the triangular identities, $\cU\cD(\beta_Y)$ and $\beta_{\cU\cD Y}$ are both right-inverses of $\cU(\alpha_{\cD Y}),$
which is an isomorphism because $\cU$ is full and faithful. \\
Concerning the isomorphism $\Theta_{\gamma,\beta} \simeq \Theta_{\beta},$ we just recall the constructions. Fix an arrow $g \colon X \to Y$ in $\cB$ 
and define $\cI_g \colon \Theta_{\gamma,\beta}(g) \to \Theta_{\beta}(g)$ by
$$\cI_g(\lambda \colon \cU\cC X \to \cU\cD Y) = (\lambda \cdot \gamma_X \colon X \to \cU\cC X \to \cU\cD Y)$$
In the opposite direction, define $\cI_g^{-1} \colon \Theta_{\beta}(g) \to \Theta_{\gamma,\beta}(g)$ by
$$\cI_g^{-1}(\psi \colon X \to \cU\cD Y) = (\cU(\delta_{\cD Y}) \cdot \cU\cC(\psi) \colon \cU\cC X \to \cU\cC\cU\cD Y \to \cU\cD Y)$$
The proof of the isomorphism $\Theta_{\gamma,\beta} \simeq \Theta_{\beta}$ is dual to the previous case.
\end{proof}

\begin{Text}\label{TextCUD}{\rm
Putting together Remark \ref{RemTriv}.2 and Proposition \ref{PropCUD}.2, we get that, in the situation of Proposition \ref{PropCUD}, 
the unit $\gamma_B$ of $\cC \dashv \cU$ is an epimorphism for all $B \in \cB$ if and only if the counit $\beta_B$ of $\cU \dashv \cD$ 
is a monomorphism for all $B \in \cB.$ 
}\end{Text}

\begin{Remark}\label{RemSubcatPrepoint}{\rm
We can complete Proposition \ref{PropCUD} by observing that the three isomorphic structures of nullhomotopies $\Theta_{\beta}, \Theta_{\gamma}$
and $\Theta_{\gamma,\beta}$ are retracts, in $\Null(\cB),$ of the structure $\Theta_{\cA}$ of Example \ref{ExNullSubcat}. Moreover, the induced ideals 
of arrows $\cZ_1(\Theta_{\beta}), \cZ_1(\Theta_{\gamma}), \cZ_1(\Theta_{\gamma,\beta})$ and $\cZ_1(\Theta_{\cA})$ are equal. 
In particular, the equality $\cZ_1(\Theta_{\beta}) = \cZ_1(\Theta_{\gamma})$ means that Proposition \ref{PropCUD} is the non-discrete generalization of the annihilation 
property stated in Lemma 1.2(c) in \cite{GJ}.
}\end{Remark} 

\begin{proof}
Let us check that $\Theta_{\gamma}$ is a retract of $\Theta_{\cA}.$ For an arrow $g \colon X \to Y$ in $\cB,$ put
$$\cI_g \colon \Theta_{\gamma}(g) \to \Theta_{\cA}(g) \;,\; \cI_g(\varphi \colon \cU\cC X \to Y) = (\gamma_X, \cC X, \varphi)$$
$$\cI_g^* \colon \Theta_{\cA}(g) \to \Theta_{\gamma}(g) \;,\; \cI_g^*(g_1,A,g_2) = g_2 \cdot g_1'$$
where $g_1' \colon \cU\cC X \to \cU A$ is the unique arrow such that $g_1' \cdot \gamma_X = g_1.$ The condition $\cI_g^*(\cI_g(\varphi)) = \varphi$
is easy to check. \\
It remains to prove that $\cZ_1(\Theta_{\gamma}) = \cZ_1(\Theta_{\cA}).$
We have to show that, for any arrow $g$ in $\cB,$ one has that $\Theta_{\gamma}(g) \neq \emptyset$ if and only if $\Theta_{\cA}(g) \neq \emptyset \colon$
this is obvious because $\Theta_{\gamma}(g)$ is a retract of $\Theta_{\cA}(g).$
\end{proof}

The rest of this section is devoted to two examples. Both are special cases of the situation described in Proposition \ref{PropCUD}.

\begin{Example}\label{ExArr}{\rm
For a given category $\cA,$ we denote by $\Arr(\cA)$ the category whose objects are the arrows $x \colon X \to X_0$ of $\cA$ and whose arrows 
are pairs of arrows $(g,g_0)$ in $\cA$ such that the following diagram commutes
$$\xymatrix{X \ar[r]^{g} \ar[d]_{x} & Y \ar[d]^{y} \\
X_0 \ar[r]_{g_0} & Y_0}$$
We will use the notation $(g,g_0) \colon (X,x,X_0) \to (Y,y,Y_0).$ There are three functors:
\begin{enumerate}
\item[-] the domain functor $\cD \colon \Arr(\cA) \to \cA$ defined by 
$$\cD((g,g_0) \colon (X,x,X_0) \to (Y,y,Y_0)) = (g \colon X \to Y)$$
\item[-] the codomain functor $\cC \colon \Arr(\cA) \to \cA$ defined by 
$$\cC((g,g_0) \colon (X,x,X_0) \to (Y,y,Y_0)) = (g_0 \colon X_0 \to Y_0)$$
\item[-] the unit functor $\cU \colon \cA \to \Arr(\cA)$ defined by
$$\cU(g \colon X \to Y) = ((g,g) \colon (X,\id_X,X) \to (Y,\id_Y,Y))$$
\end{enumerate}
The functor $\cU$ is full and faithful. Moreover, these three functors form a string of adjunctions
$$\xymatrix{ \cA \ar[rr]|-{\cU} & & \Arr(\cA) \ar@<-1.5ex>[ll]_-{\cC} \ar@<1.5ex>[ll]^-{\cD} }
\;\;\;\;\; \cC \dashv \cU \dashv \cD$$
The unit $\gamma_{(X,x,X_0)} \colon (X,x,X_0) \to \cU\cC(X,x,X_0)$ of $\cC \dashv \cU$ is
$\xymatrix{X \ar[r]^x \ar[d]_x & X_0 \ar[d]^{\id} \\ X_0 \ar[r]_{\id} & X_0}$ \\
The counit $\beta_{(Y,y,Y_0)} \colon \cU\cD(Y,y,Y_0) \to (Y,y,Y_0)$ of $\cU \dashv \cD$ is
$\xymatrix{Y \ar[r]^{\id} \ar[d]_{\id} & Y \ar[d]^y \\ Y \ar[r]_y & Y_0}$ \\
The three isomorphic structures of nullhomotopies on $\Arr(\cA)$ induced by the adjunctions $\cC \dashv \cU \dashv \cD$ as in 
Proposition \ref{PropCUD} will be denoted by $\cH(\cA).$ Explicitly, for an arrow $(g,g_0) \colon (X,x,X_0) \to (Y,y,Y_0),$ we have:
$$\cH(\cA)(g,g_0) = \{ \lambda \colon X_0 \to Y \mid \lambda \cdot x = g \mbox{ and } y \cdot \lambda = g_0 \}$$ 
Indeed, the commutativity of both triangles in 
$$\xymatrix{X \ar[r]^{g} \ar[d]_{x} & Y \ar[d]^{y} \\
X_0 \ar[r]_{g_0} \ar[ru]^{\lambda}& Y_0}$$
is clearly equivalent to the equation $\beta_{(Y,y,Y_0)} \cdot \cU(\lambda) \cdot \gamma_{(X,x,X_0)} = (g,g_0)$ in $\Arr(\cA),$ that is
$$\xymatrix{X \ar[r]^x \ar[d]_x & X_0 \ar[d]_{\id} \ar[r]^{\lambda} & Y \ar[d]^{\id} \ar[r]^{\id} & Y \ar[d]^y & \ar@{}[d]|{=} & X \ar[d]_x \ar[r]^g & Y \ar[d]^y \\
X_0 \ar[r]_{\id} & X_0 \ar[r]_{\lambda} & Y \ar[r]_y & Y_0 & & X_0 \ar[r]_{g_0} & Y_0}$$
Finally, in the situation
$$\xymatrix{W \ar[r]^f \ar[d]_w & X \ar[r]^g \ar[d]_x & Y \ar[r]^h \ar[d]^y & Z \ar[d]^z \\
W_0 \ar[r]_{f_0} & X_0 \ar[r]_{g_0} \ar[ru]^{\lambda} & Y_0 \ar[r]_{h_0} & Z_0 }$$
we have $(h,h_0) \circ \lambda \circ (f,f_0) = h \cdot \lambda \cdot f_0.$
}\end{Example}

\begin{Example}\label{ExXMod}{\rm
The second main example is in fact a variant of the first one. We take as category $\cA$ the category $\Grp$ of groups and as category
$\cB$ the category $\XMod$ of crossed modules (see \cite{WCM}). We will denote an object in $\XMod$ by $(X,x,X_0,\ast),$ where $x \colon X \to X_0$
is the group morphism and $\ast \colon X_0 \times X \to X$ is the action. The string of adjunctions
$$\xymatrix{ \Grp \ar[rr]|-{\cU} & & \XMod \ar@<-1.5ex>[ll]_-{\cC} \ar@<1.5ex>[ll]^-{\cD} }
\;\;\;\;\; \cC \dashv \cU \dashv \cD$$
is essentially as in Example \ref{ExArr}. In short, $\cD(X,x,X_0,\ast)=X, \cC(X,x,X_0,\ast)=X_0,$ and $\cU(X) = (X,\id_X,X,\conj),$ 
where $\conj \colon X \times X \to X$ is the conjugation action of $X$ on itself. The unit $\gamma_{(X,x,X_0,\ast)}$ of $\cC \dashv \cU$ 
and the counit $\beta_{(X,x,X_0,\ast)}$ of $\cU \dashv \cD$ are as in Example \ref{ExArr}, just check that they are crossed module 
morphisms and that the universal properties restrict from $\Arr(\Grp)$ to $\XMod.$ It follows that the structure of nullhomotopies induced 
on $\XMod$ by the string of adjunctions $\cC \dashv \cU \dashv \cD$ is the same as in Example \ref{ExArr}.
}\end{Example} 

\section{Homotopy kernels}\label{SecHKHC}

The natural notion of ``higher dimensional limit'' in a category equipped with a structure of nullhomotopies is the one of (strong) 
homotopy kernel. As far as we know, the first place where (a variant of) this notion has been introduced is \cite{MM}.

\begin{Definition}\label{DefHKer}{\rm
Let $g \colon X \to Y$ be an arrow in a category with nullhomotopies $(\cB,\Theta).$
\begin{enumerate}
\item A {\it homotopy kernel} of $g$ with respect to $\Theta$ is a triple 
$$\cN(g) \in \cB, n_g \colon \cN(g) \to X, \nu_g \in \Theta(n_g \cdot g)$$
such that, for any other triple of the form 
$$W \in \cB, f \colon W \to X, \varphi \in \Theta(f \cdot g)$$
there exists a unique arrow $f' \colon W \to \cN(g)$ such that $n_g \cdot f' = f$ and $\nu_g \circ f' = \varphi$
$$\xymatrix{\cN(g) \ar[rd]_{n_g} \ar@{-->}@/^1pc/[rrd]^{} & \ar@{}[d]|{\Downarrow}_{\nu_g} \\
& X \ar[r]^{g} & Y \\
W \ar[uu]^{f'} \ar[ru]^{f} \ar@{-->}@/_1pc/[rru]_{} & \ar@{}[u]^{\varphi}|{\Uparrow} }$$
\item A homotopy kernel $(\cN(g),n_g,\nu_g)$ is {\it strong} if, for any triple of the form 
$$W \in \cB, f \colon W \to \cN(g), \varphi \in \Theta(n_g \cdot f)$$
such that $g \circ \varphi = \nu_g \circ f,$ there exists a unique nullhomotopy $\varphi' \in \Theta(f)$ such that $n_g \circ \varphi' = \varphi$
$$\xymatrix{W \ar[rr]^-{f} \ar@/_1.5pc/@{-->}[rr]_{}^{\varphi' \; \Uparrow} \ar@/^1.8pc/@{-->}[rrrr]^{}_{\varphi \; \Downarrow} & & 
\cN(g) \ar[rr]^-{n_g} \ar@/_1.5pc/@{-->}[rrrr]_{}^{\nu_g \; \Uparrow} & & X \ar[rr]^{g} & & Y  }$$
\end{enumerate}
}\end{Definition}

\begin{Text}\label{TextHCoker}{\rm
The notion of {\it (strong) homotopy cokernel} with respect to $\Theta$ is dual of the notion of (strong) homotopy kernel and it will be needed
later in this paper. For the homotopy cokernel of an arrow $g \colon X \to Y,$ we adopt the notation 
$$\xymatrix{X \ar[rr]_{g} \ar@/^1.8pc/@{-->}[rrrr]^{}_{\theta_g \; \Downarrow} & & Y \ar[rr]_-{q_g} & & \cQ(g) }$$
We will often write $\Theta$-kernel and $\Theta$-cokernel instead of homotopy kernel and homotopy cokernel with respect 
to the structure $\Theta.$ In this section (in fact, throughout all the paper with the only exception of Section \ref{SecHTT}), we develop the 
theory for $\Theta$-kernels, but everything can be obviously dualized to $\Theta$-cokernels.
}\end{Text} 

\begin{Text}\label{TextUniq}{\rm 
The homotopy kernel of an arrow is determined by its universal property uniquely up to a unique isomorphism. Moreover, if an arrow has 
two (necessarily isomorphic) homotopy kernels and one of them is strong, the other one also is strong. 
}\end{Text}

\begin{Remark}\label{RemCanc}{\rm
Homotopy kernels satisfy a cancellation property. 
\begin{enumerate}
\item In the situation depicted by the following diagram, if $n_g \cdot h = n_g \cdot k$ and $\nu_g \circ h = \nu_g \circ k,$ then $h=k.$
$$\xymatrix{Z \ar@<-0.5ex>[rr]_{k} \ar@<0.5ex>[rr]^{h} & & \cN(g) \ar[rr]_-{n_g} \ar@/^1.8pc/@{-->}[rrrr]^{}_{\nu_g \; \Downarrow} & & X \ar[rr]_{g}  & & Y}$$
\item When the structure of nullhomotopies is discrete, the condition $\nu_g \circ h = \nu_g \circ k$ follows from the condition $n_g \cdot h = n_g \cdot k,$ 
so that the cancellation property above reduces to the fact that $n_g \colon \cN(g) \to X$ is a monomorphism. 
\end{enumerate} 
}\end{Remark}

\begin{Text}\label{TextTrivOrth}{\rm
In the very general context of categories with nullhomotopies, homotopy kernels do not have strong classification 
properties as, for example, usual kernels have in abelian categories. Nevertheless, in Lemma \ref{LemmaTrivOrth} we list 
some simple facts which will be useful in the rest of this paper. We start with two points of terminology in a category with 
nullhomotopies $(\cB,\Theta) \colon$
\begin{enumerate}
\item[-] An object $X \in \cB$ is $\Theta$-{\it trivial} if $\Theta(\id_X) \neq \emptyset.$
\item[-] Given an ordered pair of objects $(T,F) \in \cB \times \cB,$ we say that $T$ is $\Theta$-{\it orthogonal} to $F,$ and write $T \perp F,$ 
if $\Theta(h) = \{ \ast \}$ for every arrow $h \colon T \to F.$
\end{enumerate}
Observe that
\begin{enumerate}
\item any retract of a $\Theta$-trivial object is $\Theta$-trivial,
\item if an arrow $g$ factorizes through a $\Theta$-trivial object, then $\Theta(g) \neq \emptyset,$
\item if $T$ is isomorphic to $T'$ and $F$ is isomorphic to $F'$ and $T \perp F,$ then $T' \perp F',$
\item if $X \perp X,$ then $X$ is $\Theta$-trivial.
\end{enumerate} 
}\end{Text}

\begin{proof}
2. Consider two arrows $a \colon A \to X$ and $b \colon X \to B$ with $X$ a $\Theta$-trivial object. If $\lambda \in \Theta(\id_X),$
then $b \circ \lambda \circ a \in \Theta(b \cdot a).$
\end{proof} 

\begin{Text}\label{TextTermTriv}{\rm
Observe that the terminology in Definition \ref{DefIdeal}.2 is coherent with the one in \ref{TextTrivOrth}: 
if $\cZ_1(\Theta)$ is the ideal associated with $\Theta,$ to be $\Theta$-trivial is the same as to be $\cZ_1(\Theta)$-trivial.
}\end{Text} 

\begin{Text}\label{TextCarTrivArr}{\rm
Observe also that, using the notion of closed ideal, we can partially invert the implication in \ref{TextTrivOrth}.2: if the ideal 
$\cZ_1(\Theta)$ is closed and if $\Theta(g) \neq \emptyset,$ then $g \in \cZ_1(\Theta)$ and therefore it factorizes through 
some $\Theta$-trivial object.
}\end{Text} 

\begin{Lemma}\label{LemmaTrivOrth}
Consider a homotopy kernel in a category with nullhomotopies $(\cB,\Theta) \colon$
$$\xymatrix{\cN(g) \ar[rr]_{n_g} \ar@/^1.8pc/@{-->}[rrrr]^{}_{\nu_g \; \Downarrow} & & X \ar[rr]_-{g} & & Y }$$
\begin{enumerate}
\item $\Theta(g) \neq \emptyset$ if and only if $n_g$ is a split epimorphism. If, moreover, $\cN(g) \perp Y,$ then $n_g$ is an isomorphism.
\item If $\cN(g)$ is a strong $\Theta$-kernel and if $g$ is an isomorphism, then $\cN(g)$ is $\Theta$-trivial.
\item If $X$ is $\Theta$-trivial and if $\cN(g) \perp Y,$ then $n_g$ is an isomorphism and $\cN(g)$ is $\Theta$-trivial.
\item If $Y$ is $\Theta$-trivial and if $\cN(g) \perp Y,$ then $n_g$ is an isomorphism.
\end{enumerate} 
\end{Lemma}

\begin{proof}
1. If $\lambda \in \Theta(g) = \Theta(\id_X \cdot g),$ then, by the universal property of the $\Theta$-kernel, there exists a unique 
$a \colon X \to \cN(g)$ such that $n_g \cdot a = \id_X$ and $\nu_g \circ a = \lambda.$ The first condition on $a$ already gives that $n_g$ 
is a split epimorphism. Moreover, $n_g \cdot a \cdot n_g = \id_X \cdot n_g = n_g \cdot \id_{\cN(g)}.$ If $\cN(g) \perp Y,$ we also have
$\nu_g \circ a \cdot n_g = \nu_g \circ \id_{\cN(g)},$ because both nullhomotopies are in $\Theta(g \cdot n_g)$ which is a singleton. 
By Remark \ref{RemCanc}, we get $a \cdot n_g = \id_{\cN(g)}$ and we are done. Conversely, if there exists an arrow $i \colon X \to \cN(g)$ 
such that $n_g \cdot i = \id_X,$ then $\nu_g \circ i \in \Theta(g \cdot n_g \cdot i) = \Theta(g \cdot \id_X) = \Theta(g).$ \\
2. If $g$ is an isomorphism, then $g^{-1} \circ \nu_g \in \Theta(g^{-1} \cdot g \cdot n_g) = \Theta(n_g \cdot \id_{\cN(g)}).$ Moreover,
$g \circ (g^{-1} \circ \nu_g) = \nu_g \circ \id_{\cN(g)}.$ Since $\cN(g)$ is a strong $\Theta$-kernel, we get a unique $\lambda \in \Theta(\id_{\cN(g)})$
such that $n_g \circ \lambda = g^{-1} \circ \nu_g.$ Therefore, $\Theta(\id_{\cN(g)}) \neq \emptyset.$ \\
3. If there is a nullhomotopy $\lambda \in \Theta(\id_X),$ then $g \circ \lambda \in \Theta(g)$ and, by point 1, $n_g$
is an isomorphism. This implies that $\cN(g)$ is $\Theta$-trivial because $X$ is $\Theta$-trivial. \\
4. If there is a nullhomotopy $\lambda \in \Theta(\id_Y),$ then $\lambda \circ g \in \Theta(g)$ and we can apply point 1.
\end{proof} 

\begin{Corollary}\label{CorDiscrClStrong}
Let $(\cB,\Theta)$ be a category with a structure of nullhomotopies. Assume that $\Theta$ is discrete 
and that the ideal $\cZ_1(\Theta)$ is closed.
\begin{enumerate}
\item If an isomorphism has a $\Theta$-kernel, then the object part of the $\Theta$-kernel is $\Theta$-trivial. 
\item If $\Theta$-kernels exist in $\cB,$ then they are strong. 
\end{enumerate}
\end{Corollary} 

\begin{proof}
1. Consider a $\Theta$-kernel of an isomorphism $g$
$$\xymatrix{\cN(g) \ar[r]^-{n_g} & X \ar[r]^-{g} & Y }$$
Since $g \cdot n_g \in \cZ_1(\Theta),$ then $n_g = g^{-1} \cdot g \cdot n_g \in \cZ_1(\Theta).$ Since $\cZ_1(\Theta)$ is closed, there exists a 
factorization $n_g = b \cdot a \colon \cN(g) \to Z \to X,$ where $Z$ is a $\Theta$-trivial object. This means that $\id_Z \in \cZ_1(\Theta).$ 
and then $g \cdot b = g \cdot b \cdot \id_Z \in \cZ_1(\Theta).$  By the universal property of the $\Theta$-kernel, there exists a unique 
arrow $b' \colon Z \to \cN(g)$ such that $n_g \cdot b' = b.$ Therefore, $n_g \cdot b' \cdot a = b \cdot a = n_g.$ This implies that 
$b' \cdot a = \id_{\cN(g)},$ because $n_g$ is a monomorphism (see Remark \ref{RemCanc}.2). We have proved that $\cN(g)$ 
is a retract of $Z$ and, by \ref{TextTrivOrth}.1, we can conclude that $\cN(g)$ is $\Theta$-trivial. \\
2. Consider the following diagram
$$\xymatrix{W \ar[r]^-{f} & \cN(h) \ar[r]^-{n_h} & X \ar[r]^-{h} & Y }$$
Since $\Theta$ is discrete, to prove that the $\Theta$-kernel is strong amounts to proving the following implication: 
if $n_h \cdot f \in \cZ_1(\Theta)$ then $f \in \cZ_1(\Theta).$ Let $n_X \colon \cN(\id_X) \to X$ be the $\Theta$-kernel of the identity arrow on $X.$
Since $n_X = \id_X \cdot n_X \in \cZ_1(\Theta),$ also $h \cdot n_X \in \cZ_1(\Theta).$ By the universal property of $\cN(h),$ we get a unique
arrow $h' \colon \cN(\id_X) \to \cN(h)$ such that $n_h \cdot h' = n_X.$ On the other hand, $\id_X \cdot n_h \cdot f = n_h \cdot f \in \cZ_1(\Theta),$
so that, by the universal property of $\cN(\id_X),$ there exists a unique arrow $f' \colon W \to \cN(\id_X)$ such that $n_X \cdot f' = n_h \cdot f.$
Therefore, $n_h \cdot h' \cdot f' = n_X \cdot f' = n_h \cdot f.$ This implies that $h' \cdot f' = f$ because $n_h$ is a monomorphism
(see Remark \ref{RemCanc}.2). Finally, $f = h' \cdot f' = h' \cdot \id_{\cN(\id_X)} \cdot f' \in \cZ_1(\Theta)$ because, by the previous point,
the $\Theta$-kernel $\cN(\id_X)$ is $\Theta$-trivial, that is, $\id_{\cN(\id_X)} \in \cZ_1(\Theta).$
\end{proof}

\begin{Remark}\label{RemTruvOrth}{\rm 
In the situation of Lemma \ref{LemmaTrivOrth}, if the structure $\Theta$ is discrete, then the first point can be improved 
using Remark \ref{RemCanc}.2 and gives that $\Theta(g) \neq \emptyset$ if and only if $n_g$ is an isomorphism. 
This result appears also as Lemma 2.4 in \cite{FFG}, where it is expressed using kernels relative to an ideal of arrows 
(see Definition \ref{DefZK}).
}\end{Remark}  

\section{Existence of homotopy kernels}\label{SecExistHK}

The main result of this section is to establish a sufficient condition for the existence of homotopy kernels 
in a category equipped with a structure of nullhomotopies (Proposition \ref{PropExistenceHKHC}). This result is then specialized to 
pre-pointed categories and, in particular, to categories of arrows. We follow the same lines as done in \cite{JMMVFibr} for homotopy 
pullbacks and refine some results from \cite{GJ}. We recall from \cite{EVhcc} an auxiliary definition about nullhomotopies and (categorical) 
pullbacks.

\begin{Notation}\label{NotPBPO}{\rm 
The factorizations of a commutative square $x \cdot f = y \cdot g$ through the pullback will be written as
$$\xymatrix{A \ar@/^1.5pc/[rrrrd]^{g} \ar[rrd]^{\langle f,g \rangle} \ar@/_2.0pc/[rrdd]_{f} \\
& & B\times_{x,y}C \ar[rr]_-{x'} \ar[d]^{y'} & & C \ar[d]^{y} \\
& & B \ar[rr]_{x} & & D}$$
}\end{Notation}

\begin{Definition}\label{DefStrongPBPO}{\rm 
(Using notation \ref{NotPBPO}.) Let $(\cB,\Theta)$ be a category with nullhomotopies.
A pullback $B \times_{x,y} C$ in $\cB$ is {\it strong} with respect to the structure $\Theta$ (or $\Theta$-strong) if, given two nullhomotopies 
$\varphi \in \Theta(f)$ and $\psi \in \Theta(g)$ such that $x \circ \varphi = y \circ \psi,$ there exists a unique nullhomotopy 
$\langle \varphi,\psi \rangle \in \Theta(\langle f,g \rangle)$ such that $y' \circ \langle \varphi,\psi \rangle = \varphi$ and $x' \circ \langle \varphi,\psi \rangle = \psi.$
}\end{Definition}

\begin{Proposition}\label{PropExistenceHKHC}
Let $(\cB,\Theta)$ be a category with nullhomotopies.
 If $\cB$ has (strong) $\Theta$-kernels of identity arrows and ($\Theta$-strong) pullbacks, then $\cB$ has all $\Theta$-kernels
(and they are strong).
\end{Proposition}

\begin{proof}
Consider the following diagrams, the one on the left being a $\Theta$-kernel (where we write $n$ instead of $n_{\id}$ and $\nu$ 
instead of $\nu_{\id}$) and the one on the right being a pullback
$$\xymatrix{\cN(\id) \ar[rr]_-{n} \ar@/^1.8pc/@{-->}[rrrr]^{}_{\nu \; \Downarrow} & & Y \ar[rr]_-{\id} & & Y }
\;\;\;\;\;\;
\xymatrix{\cN(\id) \times_{n,g}X \ar[rr]^-{n'} \ar[d]_{g'} & & X \ar[d]^{g} \\ \cN(\id) \ar[rr]_-{n} & & Y }$$
Observe that $\nu \circ g' \in \Theta(n \cdot g') = \Theta(g \cdot n').$ We are going to prove that the $\Theta$-kernel of $g$ is
$$\xymatrix{\cN(\id)\times_{n,g}X \ar[rr]_-{n'} \ar@/^1.8pc/@{-->}[rrrr]^{}_{\nu \circ g' \; \Downarrow} & & X \ar[rr]_-{g} & & Y }$$
For this, consider the following situation
$$\xymatrix{W \ar[rr]_-{f} \ar@/^1.8pc/@{-->}[rrrr]^{}_{\varphi \; \Downarrow} & & X \ar[rr]_-{g} & & Y }$$
Since $\varphi \in \Theta(g \cdot f) = \Theta(\id_Y \cdot g \cdot f),$ the universal property of the $\Theta$-kernel of $\id_Y$ produces a unique 
arrow $\bar{f} \colon W \to \cN(\id)$ such that $n \cdot \bar{f} = g \cdot f$ and $\nu \circ \bar{f} = \varphi.$ Because of the first
condition on $\bar{f},$ we can apply the universal property of the pullback and we get a unique arrow $f' \colon W \to \cN(\id)\times_{n,g}X$ such that
$g' \cdot f' = \bar{f}$ and $n' \cdot f' = f.$ It follows that $\nu \circ g' \cdot f' = \nu \circ \bar{f} = \varphi,$ so that $f'$ is a factorization
of $(f,\varphi)$ through $(n',\nu \circ g').$ In order to prove the uniqueness of such a factorization, consider an arrow 
$h \colon W \to \cN(\id)\times_{n,g}X$ such that $n' \cdot h = f$ and $\nu \circ g' \cdot h = \varphi.$ To check that $h = f'$ it is enough 
to check that $g' \cdot h = \bar{f}.$ This is true because $n \cdot g' \cdot h = g \cdot n' = g \cdot f.$ \\
Now we move on to the strong case. Consider the following situation
$$\xymatrix{W \ar[rr]^-{f} \ar@/^1.8pc/@{-->}[rrrr]^{}_{\varphi \; \Downarrow} & & 
\cN(\id)\times_{n,g}X \ar[rr]^-{n'} \ar@/_1.5pc/@{-->}[rrrr]_{}^{\nu \circ g' \; \Uparrow} & & X \ar[rr]^{g} & & Y  }$$
where the nullhomotopy $\varphi$ is such that $g \circ \varphi = \nu \circ g' \cdot f.$ Observe that $g \circ \varphi \in \Theta(g \cdot n' \cdot f) 
= \Theta(n \cdot g' \cdot f)$ and $\id_Y \cdot g \circ \varphi = \nu \circ g' \cdot f.$ Therefore, since the $\Theta$-kernel of $\id_Y$ is strong,
there exists a unique nullhomotopy $\bar{\varphi} \in \Theta(f \cdot g')$ such that $n \circ \bar{\varphi} = g \circ \varphi.$ Since the pullback
$\cN(\id)\times_{n,g}X$ is $\Theta$-strong, we get a unique nullhomotopy $\varphi' \in \Theta(f)$ such that $g' \circ \varphi' = \bar{\varphi}$ and
$n' \circ \varphi' = \varphi.$ As far as the uniqueness of the factorization $\varphi'$ is concerned, let $\psi \in \Theta(f)$ be a nullhomotopy 
such that $n' \circ \psi = \varphi.$ To check that $\psi = \varphi',$ it is enough to check that $g' \circ \psi = \bar{\varphi}.$ This is true 
because $n \cdot g' \circ \psi = g \cdot n' \circ \psi = g \circ \varphi.$
\end{proof}

\begin{Lemma}\label{LemmaUnitHCoker}
Let $\beta \colon \cR \Rightarrow \cId \colon \cB \to \cB$ be a pre-radical in a category $\cB$ and $\Theta_{\beta}$ the associated structure 
of nullhomotopies as in \ref{TextRad}.1. For any object $X \in \cB,$ the following diagram is a $\Theta_{\beta}$-kernel of $\id_X$
$$\xymatrix{\cR X \ar[rr]_{\beta_X} \ar@/^1.8pc/@{-->}[rrrr]^{}_{\id_{\cR X} \; \Downarrow} & & X \ar[rr]_-{\id_X} & & X }$$
Moreover, if $\cR$ is an idempotent comonad and $\beta$ is its counit, then the $\Theta_{\beta}$-kernel described above is strong.
\end{Lemma}

\begin{proof}
Explicitly, we assert that the $\Theta_{\beta}$-kernel of $\id_X$ is
$$\xymatrix{ & \cR X \ar[rd]^{\beta_X} \\ \cR X \ar[ru]^{\id} \ar[r]_{\beta_X} & X \ar[r]_{\id} & X }$$
Indeed, a nullhomotopy
$$\xymatrix{W \ar[rr]_{f} \ar@/^1.8pc/@{-->}[rrrr]^{}_{\psi \; \Downarrow} & & X \ar[rr]_-{\id_X} & & X }$$
amounts to an arrow $\psi \colon W \to \cR X$ such that $\beta_X \cdot \psi = f.$ As factorization $f' \colon W \to \cR X$ of $(f,\psi)$
through the $\Theta_{\beta}$-kernel we can take $\psi$ itself. This is the unique possible choice because the condition $\nu_{\id_X} \circ f' = \psi$ 
precisely means $f' = \psi.$ \\
Assume now that $\cR$ is an idempotent comonad and $\beta$ is its counit, and consider a nullhomotopy $\psi$ compatible with the 
nullhomotopy $\id_{\cR X}$ as in the following diagram
$$\xymatrix{W \ar[rr]^-{f} \ar@/_1.8pc/@{-->}[rrrr]^{\psi \; \Uparrow} & & \cR X \ar[rr]_{\beta_X} \ar@/^1.8pc/@{-->}[rrrr]^{}_{\id_{\cR X} \; \Downarrow} 
& & X \ar[rr]_-{\id_X} & & X }$$
Therefore, the arrow $\psi \colon W \to \cR X$ is such that $\beta_X \cdot \psi = \beta_X \cdot f$ and $\cR(\id_X) \cdot \psi = \id_{\cR X} \cdot f.$
We get $\psi = f$ and we have to find a unique arrow $\psi' \colon W \to \cR\cR X$ such that $\beta_{\cR X} \cdot \psi' = f$ and $\cR(\beta_X) \cdot \psi' = \psi.$
If we put $\psi' = \sigma_X \cdot f,$ where $\sigma \colon \cR \Rightarrow \cR\cR$ is the comultiplication of the comonad, both conditions are satisfied.
Indeed, $\beta_{\cR X} \cdot \sigma_X \cdot f = \id_{\cR X} \cdot f = f$ and $\cR(\beta_X) \cdot \sigma_X \cdot g = \id_{\cR X} \cdot f = f = \psi.$
It remains to prove the uniqueness of $\psi'.$ For this, assume that $\beta_{\cR X} \cdot \psi' = f.$ Since the comonad is idempotent, the 
comultiplication $\sigma_X$ is an isomorphism and then $\beta_{\cR X},$ being a one-side inverse of $\sigma_X,$ is also an isomorphism. 
Finally, from $\beta_{\cR X} \cdot \psi' = f$ we get $\psi' = \beta_{\cR X}^{-1} \cdot f = \sigma_X \cdot f.$
\end{proof}

\begin{Lemma}\label{LemmaStrongPBPO}
Let $\gamma \colon \cId \Rightarrow \cS \colon \cB \to \cB$ be a pre-coradical in a category $\cB$ and $\Theta_{\gamma}$ 
the associated structure of nullhomotopies as in \ref{TextRad}.2. If $\cB$ has pullbacks, they are $\Theta_{\gamma}$-strong.
\end{Lemma}

\begin{proof}
Consider a pullback in $\cB$
$$\xymatrix{B\times_{x,y}C \ar[rr]^-{x'} \ar[d]_{y'} & & C \ar[d]^{y} \\ B \ar[rr]_{x} & & D}$$
and an arrow $h \colon W \to B\times_{x,y}C.$ Consider also two nullhomotopies $\varphi \in \Theta_{\gamma}(y' \cdot h)$ and 
$\psi \in \Theta_{\gamma}(x' \cdot h)$
such that $x \circ \varphi = y \circ \psi.$ This means that we have arrows $\varphi \colon \cS W \to B$ and $\psi \colon \cS W \to C$ such that
$\varphi \cdot \gamma_W = y' \cdot h, \psi \cdot \gamma_W = x' \cdot h$ and $x \cdot \varphi = y \cdot \psi.$ From the universal property of the pullback,
we get a unique arrow $\langle \varphi,\psi \rangle \colon \cS W \to B\times_{x,y}C$ such that $y' \cdot \langle \varphi,\psi \rangle = \varphi$ and 
$x' \cdot \langle \varphi,\psi \rangle = \psi.$ The arrow $\langle \varphi,\psi \rangle$ is in fact a nullhomotopy on $h$ because, by composing with the pullback
projections, we can check that $\langle \varphi,\psi \rangle \cdot \gamma_W= h.$ 
Moreover, $y' \circ \langle \varphi,\psi \rangle = y' \cdot \langle \varphi,\psi \rangle
= \varphi$ and $x' \circ \langle \varphi,\psi \rangle = x' \cdot \langle \varphi,\psi \rangle = \psi,$ as required. The uniqueness of such a nullhomotopy is clear: if
$\lambda \in \Theta_{\gamma}(h)$ is such that $y' \circ \lambda = \varphi$ and $x' \circ \lambda = \psi,$ we have $y' \cdot \lambda = \varphi$ and 
$x' \cdot \lambda = \psi,$ so that $\lambda = \langle \varphi,\psi \rangle.$
\end{proof}

\begin{Corollary}\label{CorHKHC}
Consider the following string of adjunctions
$$\xymatrix{ \cA \ar[rr]|{\cU} & & \cB \ar@<-1.5ex>[ll]_{\cC} \ar@<1.5ex>[ll]^{\cD} }
\;\;\;\;\; \cC \dashv \cU \dashv \cD$$
with $\cU$ full and faithful. Put on $\cB$ one of the three isomorphic structures of nullhomotopies as in Proposition \ref{PropCUD}.
If $\cB$ has pullbacks, then it has strong homotopy kernels.
\end{Corollary}

\begin{proof}
Write, as usual, $\gamma$ for the unit of the adjunction $\cC \dashv \cU$ and $\beta$ for the counit of the adjunction $\cU \dashv \cD.$
If we apply Lemma \ref{LemmaUnitHCoker} to the idempotent comonad $\cR = \cU \cdot \cD$ with counit $\beta,$ we deduce that $\cB$
has strong homotopy kernels of the identity arrows with respect to the structure $\Theta_{\beta}.$ If, moreover, we apply Lemma \ref{LemmaStrongPBPO}
to the pre-coradical $\gamma$ on $\cS = \cU \cdot \cC,$ we deduce that pullbacks in $\cB$ are strong with respect to the structure $\Theta_{\gamma}.$
Since, by Proposition \ref{PropCUD}, $\Theta_{\beta}$ and $\Theta_{\gamma}$ are isomorphic structures, we can apply Proposition \ref{PropExistenceHKHC}
to conclude that $\cB$ has strong homotopy kernels. 
\end{proof}

\begin{Remark}\label{RemExplHKHC}{\rm
Consider the situation of Corollary \ref{CorHKHC}.
\begin{enumerate}
\item Explicitly, the strong homotopy kernel of an arrow $g \colon X \to Y$ in $\cB$ can be described as follows: write 
$\beta$ for the counit of the adjunction $\cU \dashv \cD$ and consider the pullback
$$\xymatrix{ & \cU\cD Y \ar[rd]^{\beta_Y} \\
X \times_{g,\beta_Y} \cU\cD Y \ar[r]_-{\beta_Y'} \ar[ru]^{g'}  & X \ar[r]_-{g} & Y }$$
then $\cN(g) = X \times_{g,\beta_Y} \cU\cD Y,$ $n_g = \beta_Y'$ and $\nu_g = g'.$ 
\item The above construction already appears in \cite{GJ} but, since nullhomotopies are not taken into account, what 
is shown in \cite{GJ} is that 
$$\xymatrix{X \times_{g,\beta_Y} \cU\cD Y \ar[r]_-{\beta_Y'} & X \ar[r]_-{g} & Y }$$
is a weak $\cZ_1(\Theta)$-kernel of $g$ (see Definition \ref{DefZK} for the notion of $\cZ_1$-kernel).
\end{enumerate}
}\end{Remark} 

\begin{Example}\label{ExArrHKHC}{\rm
Consider the structure of nullhomotopies $\cH(\cA)$ in $\Arr(\cA)$ induced by the string of adjunctions $\cC \dashv \cU \dashv \cD$ 
as in Example \ref{ExArr}. Since $\Arr(\cA)$ is a category of presheaves with values in $\cA,$ we can improve Corollary \ref{CorHKHC} and
we get the following fact, proved directly in \cite{EVhcc}:
\begin{enumerate}
\item[-] 
if $\cA$ has pullbacks, then $\Arr(\cA)$ has strong $\cH(\cA)$-kernels. 
\end{enumerate}
More precisely, since pullbacks in $\Arr(\cA)$ are constructed level-wise from those in $\cA,$ we obtain the following explicit 
description of $\cH(\cA)$-kernels (use the description of units from  Example \ref{ExArr} and the 
construction of homotopy kernels from the proof of Proposition \ref{PropExistenceHKHC}). Let
$$\xymatrix{X \ar[r]^{g} \ar[d]_{x} & Y \ar[d]^{y} \\
X_0 \ar[r]_{g_0} & Y_0}$$
be an arrow in $\Arr(\cA).$ 
If the pullback of $g_0$ and $y$ exists, then a $\cH(\cA)$-kernel of $(g,g_0)$ is given by
$$\xymatrix{X \ar[rr]^{\id} \ar[d]_{\langle x,g \rangle} & & X \ar[d]_<<<{x} \ar[rr]^{g} & & Y \ar[d]^{y} \\ 
X_0 \times_{g_0,y} Y \ar[rrrru]_>>>>>>>>>>>>{g_0'} \ar[rr]_{y'} & & X_0 \ar[rr]_{g_0} & & Y_0}$$
This description already appears in \cite{SnailEV, EVhcc}, with its universal property as homotopy limit, as well as in 
\cite{GJ}, with the weak universal property recalled in Remark \ref{RemExplHKHC}.2.
}\end{Example} 

\begin{Example}\label{ExHExt}{\rm
Consider, in a category $\cA,$ any pair of composable arrows
$$\xymatrix{X \ar[r]^{g} & Y \ar[r]^{h} & Z }$$
The following diagram is a $\cH(\cA)$-extension in $\Arr(\cA),$ that is, $((X,g,Y),(\id_X,h), \id_Y)$ is the $\cH(\cA)$-kernel 
of $(g,\id_Z)$ and $((Y,h,Z),(g,\id_Z),\id_Y)$ is the $\cH(\cA)$-cokernel of $(\id_X,h)$ 
$$\xymatrix{X \ar[rr]^-{\id} \ar[d]_{g} & & X \ar[rr]^-{g} \ar[d]_<<<{h \cdot g} & & Y \ar[d]^{h} \\
Y \ar[rr]_-{h} \ar[rrrru]^<<<<<<<<<<{\id} & & Z \ar[rr]_-{\id} & & Z }$$
In fact, every $\cH(\cA)$-extension in $\Arr(\cA)$ is, up to isomorphism, of this form. This fact, needed in Section \ref{SecQP}, 
will be explained in more detail in Remark \ref{RemInvConstr}.
}\end{Example}

\section{A characterization of pre-pointed categories}\label{SecCharPrep}

The aim of this section is to give a characterization of pre-pointed categories among categories with a structure of nullhomotopies. 
We start by observing that, if the structure of nullhomotopies $\Theta$ comes from a monad or a comonad, then the ideal $\cZ_1(\Theta)$ 
is closed. We state the case of a monad, the situation for a comonad is dual.

\begin{Proposition}\label{PropMonComon}
Let $\gamma \colon \cId \Rightarrow \cS \colon \cB \to \cB$ be a pre-coradical in $\cB.$ 
\begin{enumerate}
\item  An object $X \in \cB$ is $\Theta_{\gamma}$-trivial iff $\gamma_X$ is a split mono.
\item In particular, if $\cS \colon \cB \to \cB$ is a monad with unit $\gamma,$ then:
\begin{enumerate}
\item If $(A,a \colon \cS A \to A)$ is an $\cS$-algebra, then $A$ is a $\Theta_{\gamma}$-trivial object.
\item For any object $X,$ the object $\cS X$ is $\Theta_{\gamma}$-trivial. 
\item The ideal $\cZ_1(\Theta_{\gamma})$ is closed: $\cZ_1(\Theta_{\gamma}) = i \{ \cS B \mid B \in \cB \}.$
\end{enumerate}
\item In particular, if the monad $\cS \colon \cB \to \cB$ is induced by an adjunction
$$\xymatrix{\cA \ar@<-0.5ex>[rr]_-{\cU} & & \cB \ar@<-0.5ex>[ll]_-{\cC} } \;\;\;
\cC \dashv \cU \;\;\;\;\; \gamma_B \colon B \to \cU\cC B \;,\;\; \delta_A \colon \cC\cU A \to A$$
with $\cU$ full and faithful, then:
\begin{enumerate}
\item $\cZ_1(\Theta_{\gamma}) = i \{ \cU A \mid A \in \cA \}.$
\item The set of $\Theta_{\gamma}$-trivial objects coincides with the closure in $\cB$ of $\{ \cU A \mid A \in \cA \}$ under isomorphisms.
\item If $X$ is a $\Theta_{\gamma}$-trivial object, then $X \perp Y$ for every object $Y \in \cB.$
\end{enumerate} 
\end{enumerate}
\end{Proposition}

\begin{proof} 
1. The commutativity of the following diagram expresses at the same time that $\gamma_X$ is a split mono and that 
$\varphi \in \Theta_{\gamma}(\id_X),$ so that $X$ is $\Theta_{\gamma}$-trivial
$$\xymatrix{ & \cS X \ar[rd]^{\varphi} \\ X \ar[ru]^{\gamma_X} \ar[rr]_{\id} & & X }$$
2(a) If $(A,a \colon \cS A \to A)$ is a $\cS$-algebra, then $a \cdot \gamma_A = \id_A$ and we apply the previous point. \\
2(b) This follows from 2(a) because $(\cS X, \mu_X \colon \cS\cS X \to \cS X)$ is a $\cS$-algebra, $\mu$ being the 
multiplication of the monad. \\
2(c) The inclusion $\cZ_1(\Theta_{\gamma}) \subseteq i \{ \cS B \mid B \in \cB \}$ is obvious: if an arrow $g \colon X \to Y$ is in 
$\cZ_1(\Theta_{\gamma}),$ then $\Theta_{\gamma}(g)$ is non-empty, so that there exists an arrow $\varphi \colon \cS X \to Y$ 
such that $g =\varphi \cdot  \gamma_X \colon X \to \cS X \to Y.$ (By 2(b), the object $\cS X$ is $\Theta_{\gamma}$-trivial, 
so that we can already conclude that $\cZ_1(\Theta_{\gamma})$ is closed.) \\
Conversely, if an arrow $g \colon X \to Y$ can be factorized as $g = b \cdot a \colon X \to \cS B \to Y,$ we get
$$b \cdot \mu_B \cdot \cS(a) \cdot \gamma_X = b \cdot \mu_B \cdot \gamma_{\cS B} \cdot a = b \cdot a = g$$
so that $b \cdot \mu_B \cdot \cS(a) \in \Theta_{\gamma}(g)$ and then $g \in \cZ_1(\Theta_{\gamma}).$ \\
3(a) Obvious because $\cS B = \cU\cC B$ and $\cU A \simeq \cU\cC\cU A.$ \\
3(b) This is a consequence of \ref{TextRetrId}.4 and of the following standard fact: if an object $B \in \cB$ is a retract of 
an object coming from $\cA,$ say $\id_B = y \cdot x \colon B \to \cU A \to B,$ then the unit $\gamma_B \colon B \to \cU\cC B$
is an isomorphism with inverse given by $y \cdot \cU(\delta_A) \cdot \cU\cC(x).$ \\
3(c) It is enough to prove the statement when $X = \cU A$ for $A \in \cA.$ Consider an arrow $g \colon \cU A \to Y.$ Then
$\Theta_{\gamma}(g) = \{ \varphi \colon \cU\cC\cU A \to Y \mid \varphi \cdot \gamma_{\cU A} = g \}.$ But $\gamma_{\cU A}$
is an isomorphism, so that $\Theta_{\gamma}(g)$ is reduced to the element $\varphi = g \cdot \gamma_{\cU A}^{-1}.$
\end{proof}

\begin{Corollary}\label{CorTrivAdj}
Consider a string of adjunctions
$$\xymatrix{ \cA \ar[rr]|{\cU} & & \cB \ar@<-1.5ex>[ll]_{\cC} \ar@<1.5ex>[ll]^{\cD} }
\;\;\;\;\; \cC \dashv \cU \dashv \cD  \;\;\;\;\; \gamma_B \colon B \to \cU\cC B \;,\;\; \beta_B \colon \cU\cD B \to B$$
with $\cU$ full and faithful. Let $\Theta$ be the induced structure of nullhomotopies 
on $\cB.$ The $\Theta$-trivial objects are orthogonal, on  both sides, to any object of $\cB.$
\end{Corollary}

\begin{proof}
Thanks to Proposition \ref{PropMonComon}.3(b), it is enough to check the statement for the $\Theta$-trivial objects 
of the form $\cU A,$ with $A$ varying in $\cA.$ We use the isomorphisms of Proposition \ref{PropCUD}. \\
- If $g \colon \cU A \to Y,$ then $\Theta(g) \simeq \Theta_{\gamma}(g) = \{ \ast \}$ by Proposition \ref{PropMonComon}.3(c). \\
- If $g \colon X \to \cU A,$ then $\Theta(g) \simeq \Theta_{\beta}(g) = \{ \ast \}$ by the dual of Proposition \ref{PropMonComon}.3(c).
\end{proof} 

\begin{Example}\label{ExArrId}{\rm
Consider the structure of nullhomotopies $\cH(\cA)$ in $\Arr(\cA)$ induced by the string of adjunctions 
$\cC \dashv \cU \dashv \cD$ as in Example \ref{ExArr}. Write $\cZ_1(\cA)$ for the associated ideal of 
arrows in $\Arr(\cA),$ as in Lemma \ref{LemmaNullId}. The previous analysis shows that $\cZ_1(\cA)$ 
is closed: there exists a nullhomotopy $\lambda \colon X_0 \to Y$ on an arrow 
$(g,g_0) \colon (X,x,X_0) \to (Y,y,Y_0)$ iff $(g,g_0)$ factors through some $\cZ_1(\cA)$-trivial objects. 
Following \ref{PropMonComon}.3, the $\cZ_1(\cA)$-trivial objects are precisely the objects $(N,n,N_0)$ 
such that $n \colon N \to N_0$ is an isomorphism. Clearly, such an object $(N,n,N_0)$ is orthogonal, 
on both sides, to any object of $\Arr(\cA).$
}\end{Example}

\begin{Text}\label{TextCharPrep}{\rm
With the next results, we show that, up to identifying a full and faithful functor with its replete image, pre-pointed 
categories as in Proposition \ref{PropCUD} can be detected, among categories with nullhomotopies, by
the existence of some strong homotopy kernels and homotopy cokernels and by the behavior of trivial objects. 
Keep in mind Lemma \ref{LemmaUnitHCoker} and Proposition \ref{PropMonComon}.3(c), which show that the 
assumptions in Lemma \ref{LemmaPreRad} and in Proposition \ref{PropCharactPrepointed} are indeed necessary 
conditions.
}\end{Text} 

\begin{Lemma}\label{LemmaPreRad}
Let $(\cB,\Theta)$ be a category with nullhomotopies. 
If, for every object $X \in \cB,$ there exists a $\Theta$-kernel
$$\xymatrix{\cN(\id_X) \ar[rr]_{n_X} \ar@/^1.8pc/@{-->}[rrrr]^{}_{\nu_X \; \Downarrow} & & X \ar[rr]_-{\id} & & X }$$
then $n \colon \cN \Rightarrow \Id \colon \cB \to \cB$ is a pre-radical and the induced structure of nullhomotopies
$\Theta_n$ (see \ref{TextRad}.2) coincides with $\Theta.$
\end{Lemma}

\begin{proof}
The functor $\cN \colon \cB \to \cB$ of the statement sends an object $X \in \cB$ to the object part of the 
$\Theta$-kernel of the identity on $X.$ To extend it to arrows, consider an arrow $g \colon X \to Y.$ Since 
$\nu_X \in \Theta(\id_X \cdot n_X) = \Theta(n_X),$ then $g \circ \nu_X \in \Theta(g \cdot n_X) = \Theta(\id_Y \cdot g \cdot n_X).$
By the universal property of $\cN(\id_Y),$ we get a unique arrow $\cN_g \colon \cN(\id_X) \to \cN(\id_Y)$ such that
$g \cdot n_X = n_Y \cdot \cN_g.$ This gives at once the definition of the functor $\cN$ on arrows and the naturality of 
$n \colon \cN \Rightarrow \Id.$ \\
In order to construct the isomorphism $\Theta_n \simeq \Theta,$ recall that, for an arrow $g \colon X \to Y,$ we have
$\Theta_n(g) = \{ \psi \colon X \to \cN(\id_Y) \mid n_Y \cdot \psi = g \}.$ Now we put: \\
- $\Theta_n(g) \to \Theta(g) \colon \psi \mapsto \nu_Y \circ \psi$ \\
- $\Theta(g) \to \Theta_n(g) \colon \lambda \mapsto \lambda',$ where $\lambda' \colon X \to \cN(\id_Y)$ is the unique 
arrow such that $n_Y \cdot \lambda' = g$ and $\nu_Y \circ \lambda' = \lambda.$ \\
To check that these maps realize an isomorphism of nullhomotopy structures is easy.
\end{proof} 

\begin{Proposition}\label{PropCharactPrepointed}
Let $(\cB, \Theta)$ be a category with nullhomotopies and let $\cA$ be the full subcategory of $\Theta$-trivial objects.
Assume that
\begin{enumerate}
\item For every object $X \in \cB$ there exist a strong $\Theta$-kernel and a strong $\Theta$-cokernel of $\id_X$
$$\xymatrix{\cN(\id_X) \ar[rr]_{n_X} \ar@/^1.8pc/@{-->}[rrrr]^{}_{\nu_X \; \Downarrow} & & 
X \ar[rr]_-{\id} \ar@/^1.8pc/@{-->}[rrrr]^{}_{\Downarrow \; \theta_X} & & X \ar[rr]_{q_X} & & \cQ(\id_X) }$$
\item For all $A \in \cA$ and for all $X \in \cB,$ we have $A \perp X$ and $X \perp A.$
\end{enumerate}
Then $\cA$ is reflective and coreflective in $\cB$ and the induced structure of nullhomotopies on $\cB$ (see Proposition
\ref{PropCUD}) coincides with $\Theta.$
\end{Proposition} 

\begin{proof}
Since $\id_X \colon X \to X$ is an isomorphism and since its $\Theta$-kernel is strong, by Lemma \ref{LemmaTrivOrth}.2 
we have that the object $\cN(\id_X)$ is $\Theta$-trivial. Therefore, the functor $\cN \colon \cB \to \cB$ of Lemma
\ref{LemmaPreRad} factorizes through the subcategory $\cA.$ The same holds also for the functor $\cQ \colon \cB \to \cB$ 
dual of the one of Lemma \ref{LemmaPreRad}. It remains to prove that $q_X \colon X \to \cQ(\id_X)$ provides a unit 
for the adjunction $\cQ \dashv \cU$ and $n_X \colon \cN(\id_X) \to X$ provides a counit for the adjunction $\cU \dashv \cN,$
where $\cU \colon \cA \to \cB$ is the full inclusion. We do the job for the unit $q_X.$ Let $g \colon X \to A$ be an arrow with 
$A \in \cA.$ Since $A$ is $\Theta$-trivial, there exists $\lambda \in \Theta(\id_A)$ and then 
$\lambda \circ g \in \Theta(\id_A \cdot g) = \Theta(g \cdot \id_X).$ By the universal property of the $\Theta$-cokernel, there
exists a unique arrow $g' \colon \cQ(\id_X) \to A$ such that $g' \cdot q_X = g$ and $g' \circ \theta_X = \lambda.$ Consider
now another arrow $g'' \colon \cQ(\id_X) \to A$ such that $g'' \cdot q_X = g.$ To show that $g'' = g',$ we have to prove that
$g'' \circ \theta_X = \lambda,$ but $g'' \circ \theta_X$ and $\lambda$ are elements of $\Theta(g),$ which is reduced to a
singleton because $X \perp A.$ Finally, the fact that the structure of nullhomotopies induced on $\cB$ by the string of adjunctions
$\cQ \dashv \cU \dashv \cN$ coincides with the original structure $\Theta$ comes directly from Lemma \ref{LemmaPreRad}.
\end{proof} 

\section{Homotopy torsion theories}\label{SecHTT}

Torsion theories have been originally introduced in the context of abelian categories by Dickson, see \cite{Dick}. 
We refer to Chapter 1 in \cite{BO2} for the classical theory. 

\begin{Definition}\label{DefHTT}{\rm
Let $(\cB,\Theta)$ be a category with nullhomotopies.
A {\it homotopy torsion theory} in $\cB$ relative to the structure $\Theta$ (or $\Theta$-torsion theory) is given by two full subcategories
$$\cT \subseteq \cB \,,\; \cF \subseteq \cB$$
such that
\begin{enumerate}
\item[1)] Both subcategories are replete, that is, closed under isomorphisms. (Note that this condition does not depend on the structure $\Theta.$)
\item[2)] For any object $X \in \cB,$ there exists a $\Theta$-exact $(\cT,\cF)$-presentation, that is, a diagram 
$$\xymatrix{T(X) \ar[rr]_{t_X} \ar@/^1.8pc/@{-->}[rrrr]^{}_{\xi_X \; \Downarrow} & & X \ar[rr]_-{f_X} & & F(X) }$$
such that 
\begin{enumerate}
\item $T(X) \in \cT$ and the triple $(T(X), t_X, \xi_X)$ is a $\Theta$-kernel of $f_X.$
\item $F(X) \in \cF$ and the triple $(F(X), f_X, \xi_X)$ is a $\Theta$-cokernel of $t_X.$
\end{enumerate}
\item[3)] For all $T \in \cT$ and for all $F \in \cF,$ we have that $T$ is $\Theta$-orthogonal to $F$ (see \ref{TextTrivOrth}).
\end{enumerate}
}\end{Definition}

Here there are some test properties for our definition of homotopy torsion theory.

\begin{Proposition}\label{PropUniqHTT}
Let $(\cB,\Theta)$ be a category with nullhomotopies and $(\cT,\cF)$ a $\Theta$-torsion theory.
The $\Theta$-exact $(\cT,\cF)$-presentation of an object is essentially unique.
\end{Proposition}

\begin{proof}
In the situation of Definition \ref{DefHTT}, consider two $\Theta$-exact $(\cT,\cF)$-presentations of an object $X \colon$
$$\xymatrix{T(X) \ar[rr]_{t_X} \ar@/^1.8pc/@{-->}[rrrr]^{}_{\xi_X \; \Downarrow} & & X \ar[rr]_-{f_X} & & F(X) }
\;\;\;\;\;\;
\xymatrix{T \ar[rr]_{t} \ar@/^1.8pc/@{-->}[rrrr]^{}_{\xi \; \Downarrow} & & X \ar[rr]_-{f} & & F }$$
Since $T \in \cT$ and $F(X) \in \cF,$ there exists a unique $\tau \in \Theta(f_X \cdot t).$ By the universal property of the $\Theta$-kernel
$(T(X),t_X,\xi_X),$ there exists a unique arrow $a \colon T \to T(X)$ such that $t_X \cdot a = t$ and $\xi_X \circ a = \tau.$
Since $T(X) \in \cT$ and $F \in \cF,$ there exists a unique $\tau' \in \Theta(f \cdot t_X).$  By the universal property of the $\Theta$-kernel
$(T,t,\xi),$ there exists a unique arrow $a' \colon T(X) \to T$ such that $t \cdot a' = t_X$ and $\xi \circ a' = \tau'.$ Observe that 
$\tau \circ a' = \xi_X$ because they are in $\Theta(f_X \cdot t \cdot a') = \Theta(f_X \cdot t_X),$ which is reduced to a singleton. Now we have
$$t_X \cdot a \cdot a' = t \cdot a' = t_X = t_X \cdot \id_{T(X)} \;\mbox{ and }\; \xi_X \circ a \cdot a' = \tau \circ a' = \xi_X = \xi_X \circ \id_{T(X)}$$
By Remark \ref{RemCanc}, we can conclude that $a \cdot a' = \id_{T(X)}.$ Similarly, we have $a' \cdot a = \id_T.$ The construction of an
isomorphism $F(X) \simeq F$ commuting with $f_X$ and $f$ is dual.
\end{proof}

\begin{Proposition}\label{PropCarTF}
Let $(\cB,\Theta)$ be a category with nullhomotopies and $(\cT,\cF)$ a $\Theta$-torsion theory. Fix an object $X \in \cB$ 
and a $\Theta$-exact $(\cT,\cF)$-presentation of $X$ as in Definition \ref{DefHTT}.
\begin{enumerate}
\item The following conditions are equivalent:
\begin{enumerate}
\item $X \in \cF,$ 
\item $f_X \colon X \to F(X)$ is an isomorphism,
\item $\Theta(t_X) \neq \emptyset,$
\item $T \perp X$ for all $T \in \cT.$
\end{enumerate} 
\item The following conditions are equivalent:
\begin{enumerate}
\item $X \in \cT,$ 
\item $t_X \colon T(X) \to X$ is an isomorphism,
\item $\Theta(f_X) \neq \emptyset,$
\item $X \perp F$ for all $F \in \cF.$
\end{enumerate} 
\end{enumerate}
\end{Proposition}

\begin{proof}
1. (a) $\Rightarrow$ (d): This is point 3 in Definition \ref{DefHTT}. \\
(d) $\Rightarrow$ (c): If condition (d) holds, in particular $T(X) \perp X,$ so that $\Theta(t_X) = \{ \ast \} \neq \emptyset.$ \\
(c) $\Rightarrow$ (b): Since $T(X) \perp F(X)$ and since condition (c) holds, we can apply the dual of Lemma \ref{LemmaTrivOrth}.1 
to conclude that $f_X$ is an isomorphism. \\
(b) $\Rightarrow$ (a):  This follows from the fact that $\cF$ is replete. 
\end{proof} 

\begin{Proposition}\label{PropCarTFbis}
Let $(\cB,\Theta)$ be a category with nullhomotopies and $(\cT,\cF)$ a $\Theta$-torsion theory. Fix an object $X \in \cB$ 
and a $\Theta$-exact $(\cT,\cF)$-presentation of $X$ as in Definition \ref{DefHTT}. 
\begin{enumerate}
\item If $T(X)$ is $\Theta$-trivial, then $X \in \cF.$ The converse holds
\begin{enumerate}
\item if the $\Theta$-kernel $(T(X),t_X,\xi_X)$ is strong, or
\item if the ideal $\cZ_1(\Theta)$ is closed.
\end{enumerate} 
\item If $F(X)$ is $\Theta$-trivial, then $X \in \cT.$ The converse holds
\begin{enumerate}
\item if the $\Theta$-cokernel $(F(X),f_X,\xi_X)$ is strong, or
\item if the ideal $\cZ_1(\Theta)$ is closed.
\end{enumerate} 
\end{enumerate}
\end{Proposition}

\begin{proof}
2. Assume that $F(X)$ is $\Theta$-trivial and let $\varphi \in \Theta(\id_{F(X)}).$ It follows that $\varphi \circ f_X \in \Theta(f_X).$
We can apply point 2(c) of Proposition \ref{PropCarTF} to conclude that $X \in \cT.$ \\
Conversely, assume first condition (a): if $X \in \cT,$ then $t_X \colon T(X) \to X$ is an isomorphism by \ref{PropCarTF}.2. 
Since we assume that the $\Theta$-cokernel $F(X)$ is strong, we can apply the dual of Lemma \ref{LemmaTrivOrth}.2 to conclude 
that $F(X)$ is $\Theta$-trivial.\\
Assume now condition (b): if $X \in \cT,$ then $t_X \colon T(X) \to X$ is an isomorphism by \ref{PropCarTF}.2. 
We have $\xi_X \circ t_X^{-1} \in \Theta(f_X),$ so that $f_X \in \cZ_1(\Theta).$ Since $\cZ_1(\Theta)$ is closed, there exists a factorization 
$f_X = v \cdot u \colon X \to Z \to F(X)$ through some $\Theta$-trivial object $Z$ (see \ref{TextCarTrivArr}). Since $\id_Z \in \cZ_1(\Theta),$ 
then also $u = \id_Z \cdot u \in \cZ_1(\Theta)$ and then there exists a nullhomotopy $\lambda \in \Theta(u).$ From the nullhomotopy 
$\lambda \circ t_X \in \Theta(u \cdot t_X)$ and the universal property of the $\Theta$-cokernel $F(X),$ we get a unique arrow $u' \colon F(X) \to Z$ 
such that $u' \cdot f_X = u$ and $u' \circ \xi_X = \lambda \circ t_X.$ 
Observe that $v \cdot u' \cdot f_X = v \cdot u = f_X.$ The condition $v \circ u' \cdot \xi_X = \xi_X$ comes for free because $T(X) \perp F(X),$ 
so that we can apply Remark \ref{RemCanc}.1 and we get $v \cdot u' = \id_{F(X)}.$ We have proved that $F(X)$ is a retract of $Z,$ so that 
$F(X)$ is $\Theta$-trivial.
\end{proof} 

\begin{Corollary}\label{CorClosRetr}
Let $(\cB,\Theta)$ be a category with nullhomotopies and $(\cT,\cF)$ a $\Theta$-torsion theory.
The subcategories $\cT$ and $\cF$ are closed under retracts.
\end{Corollary}

\begin{proof}
Let $F \in \cF$ and consider a retract $X$ of $F,$ with $a \colon X \to F$ and $b \colon F \to X$ such that 
$b \cdot a = \id_X.$ We have to prove that $X \in \cF.$ By Proposition \ref{PropCarTF}.1, it suffices to show that 
$\Theta(t_X) \neq \emptyset.$ Since $T(X) \perp F,$ there exists a (unique) nullhomotopy $\lambda \in \Theta(a \cdot t_X).$ 
Therefore, $b \circ \lambda \in \Theta(b \cdot a \cdot t_X) = \Theta(t_X),$ and we are done.
The argument for $\cT$ is dual.
\end{proof} 

\begin{Corollary}\label{CorCarTrivObj}
Let $(\cB,\Theta)$ be a category with nullhomotopies and $(\cT,\cF)$ a $\Theta$-torsion theory.
The following conditions on an object $X \in \cB$ are equivalent:
\begin{enumerate}
\item[(a)] $X \in \cT \cap \cF,$
\item[(b)] $\Theta(\id_X) = \{ \ast \},$
\item[(c)] $X$ is $\Theta$-trivial.
\end{enumerate}
\end{Corollary} 

\begin{proof}
If $X \in \cT \cap \cF,$ then $X \perp X$ and the orthogonality condition gives $\Theta(\id_X) = \{ \ast \}.$ 
If $X$ is $\Theta$-trivial, we can apply Lemma \ref{LemmaTrivOrth}.3 and its dual to the $\Theta$-exact $(\cT,\cF)$-presentation 
of $X.$ We get that $t_X \colon T(X) \to X$ and $f_X \colon X \to F(X)$ are isomorphisms. Since $\cT$ and $\cF$ are replete, 
we are done. The implication (b) $\Rightarrow$ (c) is obvious. 
\end{proof} 

\begin{Corollary}\label{CorCarTF}
Let $(\cB,\Theta)$ be a category with nullhomotopies and $(\cT,\cF)$ a $\Theta$-torsion theory.
\begin{enumerate}
\item For any object $X \in \cF,$ it can be chosen a $\Theta$-exact $(\cT,\cF)$-presentation of the form
$$\xymatrix{T(X) \ar[r]^-{t_X} & X \ar[r]^-{\id} & X }$$
In particular, the arrow $\id_X$ has a $\Theta$-kernel whose object part lies in $\cT.$
\item For any object $X \in \cT,$ it can be chosen a $\Theta$-exact $(\cT,\cF)$-presentation of the form
$$\xymatrix{X \ar[r]^-{\id} & X \ar[r]^-{f_X} & F(X) }$$
In particular, the arrow $\id_X$ has a $\Theta$-cokernel whose object part lies in $\cF.$
\end{enumerate}
\end{Corollary} 

\begin{proof} 
1. Since the $\Theta$-cokernel is defined up to isomorphism and since $\cF$ is replete, in the $\Theta$-exact 
$(\cT,\cF)$-presentation of $X$ we can replace the $\cF$-part $f_X,$ which is an isomorphism by Proposition 
\ref{PropCarTF}, with $f_X^{-1} \cdot f_X.$
\end{proof} 

If we allow us to choose an object in an isomorphism class of objects, from Proposition \ref{PropUniqHTT} we 
get the following corollary.

\begin{Corollary}\label{CorReflect}
Let $(\cB,\Theta)$ be a category with nullhomotopies and $(\cT,\cF)$ a $\Theta$-torsion theory.
Then $\cT$ is coreflective in $\cB$ and $\cF$ is reflective in $\cB.$
\end{Corollary} 

\begin{proof} 
For any object $X \in \cB,$ just choose a $\Theta$-exact $(\cT,\cF)$-presentation
$$\xymatrix{T(X) \ar[rr]_{t_X} \ar@/^1.8pc/@{-->}[rrrr]^{}_{\xi_X \; \Downarrow} & & X \ar[rr]_-{f_X} & & F(X) }$$
Then $t_X \colon T(X) \to X$ is the counit of the coreflection of $\cB$ on $\cT.$ Dually, $f_X \colon X \to F(X)$ is 
the unit of the reflection of $\cB$ on $\cF.$  To extend to arrows, use the universal properties of the $\Theta$-kernel 
and of the $\Theta$-cokernel involved in the $\Theta$-exact $(\cT,\cF)$-presentations. For example, given 
an arrow $g \colon X \to Y,$ we get an arrow $f_Y \cdot g \cdot t_X \colon T(X) \to F(Y).$ Since $T(X) \perp F(Y),$ 
there exists a unique nullhomotopy $\lambda_g \in \Theta(f_Y \cdot g \cdot t_X).$ Now, the universal property of the
$\Theta$-cokernel of $t_X$ gives a unique arrow $F(g) \colon F(X) \to F(Y)$ such that $F(g) \cdot f_X = f_Y \cdot g$
and $F(g) \circ \xi_X = \lambda_g.$ The needed functoriality and naturality come from uniqueness. 
As far as the universal properties of the unit and of the counit are concerned, let us check the one of the unit. Consider an 
arrow $h \colon X \to F$ with $F \in \cF.$ Since $T(X) \perp F,$ there exists a unique $\lambda \in \Theta(h \cdot t_X).$
The universal property of the $\Theta$-cokernel gives a unique arrow $h' \colon F(X) \to F$ such that $h' \cdot f_X = h$
and $h' \circ \xi_X = \lambda.$ If another arrow $h'' \colon F(X) \to F$ is such that $h'' \cdot f_X = h,$ then 
$h'' \circ \xi_X \in \Theta(h \cdot t_X),$ which is a singleton. Therefore, $h'' \circ \xi_X = \lambda$ and, finally, $h'' = h'.$
\end{proof} 

\begin{Example}\label{ExCExNSub}{\rm 
To end this section, we go back to the structure of nullhomotopies $\cH(\Grp)$ in the category $\XMod$ introduced in Example \ref{ExXMod}.
Given a crossed module $(X,x,X_0,\ast),$ we can consider the factorization of $x$ through its image $I(x)$ as in the following diagram
$$\xymatrix{X \ar[rr]^-{x} \ar[rd]_{e_x} & & X_0 \\ & I(x) \ar[ru]_{m_x} }$$
It is well-known that $m_x \colon I(x) \to X_0$ is a normal subgroup, so that we get a crossed module $(I(x),m_x,X_0,\conj)$ with $X_0$ 
acting on $I(x)$ by conjugation. Moreover, $e_x \colon X \to I(x)$ is a central extension, that is, a surjective morphism with central kernel. 
We get another crossed module $(X,e_x,I(x),\ast)$ with $I(x)$ acting on $X$ with action $a \ast b = c \cdot b \cdot c^{-1},$ where $c$ is any 
element of $X$ such that $e_x(c)=a.$ We are ready to produce our first example of homotopy torsion theory. It is given by the following full
 subcategories of $\XMod \colon$
$$\cT = \mbox{the full subcategory of central extensions, i.e., surjective crossed modules}$$
$$\cF = \mbox{the full subcategory of normal subgroups, i.e., injective crossed modules}$$
For a crossed module $(X,x,X_0,\ast),$ its $\cH(\Grp)$-exact $(\cT,\cF)$-presentation is given by
$$\xymatrix{X \ar[rr]^-{\id} \ar[d]_{e_x} & & X \ar[rr]^-{e_x} \ar[d]_<<<{x} & & I(x) \ar[d]^{m_x} \\
I(x) \ar[rr]_-{m_x} \ar[rrrru]^<<<<<<<<<<{\id} & & X_0 \ar[rr]_-{\id} & & X_0 }$$
Note that the $\cH(\Grp)$-exactness is as in Examples \ref{ExArrHKHC} and \ref{ExHExt}.
}\end{Example} 

\section{Comparison with pretorsion theories}\label{SecPretorsion}

The aim of this section is to establish a complete comparison between pretorsion theories and homotopy torsion theories. 
The first step is to compare kernels and cokernels relative to an ideal (called prekernels and precokernels in \cite{FF, FFG}) 
with homotopy kernels and homotopy cokernels. 
(We do the job for kernels and we leave to the reader to dualize Definition \ref{DefZK}, Lemma \ref{LemmaZHKer} and Remark \ref{RemDiscrHK}.)
Then, we will compare pretorsion theories and homotopy torsion theories 
in Corollary \ref{CorZHTT} and Corollary \ref{CorHTTpreTTbis}. We recall from \cite{GJ, FF} the following definition.

\begin{Definition}\label{DefZK}{\rm
Let $\cZ_1$ be an ideal of arrows in a category $\cB.$ Fix an arrow $g \colon X \to Y.$ 
A $\cZ_1$-{\it kernel} of $g$ is an arrow $k_g \colon K(g) \to X$ such that $g \cdot k_g \in \cZ_1$ and which is universal with respect to this condition: 
if $f \colon W \to X$ satisfies $g \cdot f \in \cZ_1,$ then there exists a unique $f' \colon W \to K(g)$ such that $k_g \cdot f' = f.$
}\end{Definition}

Here is the expected result, whose proof is straightforward.

\begin{Lemma}\label{LemmaZHKer}
Let $(\cB,\Theta)$ be a category with nullhomotopies.
If $\Theta$ is discrete, then
$\Theta$-kernels in the sense of \ref{DefHKer} coincide with $\cZ_1(\Theta)$-kernels in the sense of \ref{DefZK},
\end{Lemma}

\begin{Remark}\label{RemDiscrHK}{\rm 
Observe that the above fact is no longer true if we start with a structure of nullhomotopies $\Theta$ on $\cB$ which is not discrete: 
if we assume that the $\Theta$-kernel and the $\cZ_1(\Theta)$-kernel of an arrow exist, then the $\cZ_1(\Theta)$-kernel is a retract 
of the $\Theta$-kernel, but in general they do not coincide. 
}\end{Remark}

\begin{Text}\label{TextZTPreT}{\rm
To make easier the comparison between homotopy torsion theories and pretorsion theories, we write down explicitly the intermediate notion of 
$\cZ_1$-torsion theory for $\cZ_1$ an ideal of arrows. In fact, a pair $(\cT,\cF)$ of full subcategories of a category $\cB$ is a pretorsion theory 
in the sense of \cite{FF,FFG} precisely when it is a $\cZ_1$-torsion theory for the closed ideal $\cZ_1 = i(\ob(\cT \cap \cF)).$
Let us observe also that, if $\cZ_1$ is closed, then $\cZ_1$-torsion theories precisely are torsion theories in multi-pointed categories
in the sense of \cite{GJ}.
}\end{Text} 

\begin{Definition}\label{DefZTT}{\rm
Let $\cZ_1$ be an ideal of arrows in a category $\cB.$
A $\cZ_1$-{\it torsion theory} in $\cB$ is given by two full replete subcategories $\cT$ and $\cF$ of $\cB$ such that
\begin{enumerate}
\item For any object $X \in \cB,$ there exists a diagram of the form
$$\xymatrix{T(X) \ar[rr]^-{t_X} & & X \ar[rr]^-{f_X} & & F(X) }$$
where $T(X) \in \cT, F(X) \in \cF, t_X$ is a $\cZ_1$-kernel of $f_X$ and $f_x$ is a $\cZ_1$-cokernel of $t_X.$
\item For any object $T \in \cT$ and $F \in \cF,$ any arrow $h \colon T \to F$ belongs to $\cZ_1.$
\end{enumerate}
}\end{Definition}

\begin{Remark}\label{RemPointed}{\rm
If the category $\cB$ has a zero object 0 and if we consider the ideal $\cZ_1(0)$ of zero arrows (that is, the arrows 
which factorize through the zero object), then $\cZ_1(0)$-kernels and $\cZ_1(0)$-cokernels coincide with kernels and 
cokernels in the usual sense. Moreover, $\cZ_1(0)$-torsion theories are nothing but the classical (not necessarily 
abelian) torsion theories as in Chapter 1 of \cite{BO2}.
}\end{Remark} 

With the next two corollaries, we complete the comparison between homotopy torsion theories and pretorsion theories. 
The first one follows easily from Lemma \ref{LemmaZHKer}.

\begin{Corollary}\label{CorZHTT}
Let $(\cB,\Theta)$ be a category with nullhomotopies. 
Consider two full subcategories $\cT$ and $\cF$ of $\cB.$ If $\Theta$ is discrete, then $(\cT,\cF)$ is a 
$\Theta$-torsion theory in the sense of \ref{DefHTT} if and only if it is a $\cZ_1(\Theta)$-torsion theory in the sense of \ref{DefZTT}.
\end{Corollary} 

\begin{Corollary}\label{CorHTTpreTTbis}
Let $\cZ_1$ be a closed ideal of arrows in a category $\cB.$ If a pair of full subcategories $(\cT,\cF)$ of $\cB$ is a $\cZ_1$-torsion theory, 
then $i(\ob(\cT \cap \cF)) = \cZ_1,$ so that $(\cT,\cF)$ is also a pretorsion theory. 
\end{Corollary} 

\begin{proof}
The arrows in $\cZ_1$ are precisely the arrows which factorize  
through some $\cZ_1$-trivial object. The arrows in $i(\ob(\cT \cap \cF))$ are precisely the arrows which factorize through some object in
$\cT \cap \cF.$ But being $\cZ_1$-trivial is the same thing as being $\Theta_{\cZ_1}$-trivial, and being $\Theta_{\cZ_1}$-trivial is, by Corollary 
\ref{CorCarTrivObj}, the same thing as being in $\cT \cap \cF.$ This shows that $i(\ob(\cT \cap \cF)) = \cZ_1,$ so that $(\cT,\cF)$ is a
$i(\ob(\cT \cap \cF))$-torsion theory. By  \ref{TextZTPreT}, we can conclude that $(\cT,\cF)$ is a pretorsion theory.
\end{proof} 

\begin{Text}\label{TextFinPreTT}{\rm
Thanks to Lemma \ref{LemmaNullId} and Lemma \ref{LemmaZHKer}, we can identify the discrete structures of nullhomotopies with the 
corresponding ideals of arrows. We can now resume the analysis developed in this section saying that pretorsion theories are exactly 
homotopy torsion theories with respect to structures which are discrete and closed. 
}\end{Text}

\section{Factorization systems are homotopy torsion theories}\label{SecFSHTT}
 
Factorization systems have a long history in category theory. We refer to Chapter 5 in \cite{BO1} for the classical theory, and to \cite{CLF} 
and the references therein for a more recent update. We adopt the name of orthogonal factorization system (the word orthogonal does not 
appear in the name used in \cite{BO1}) to underline the presence of the orthogonality condition in the definition.

\begin{Definition}\label{DefFS}{\rm
Let $\cA$ be a category. An {\it orthogonal factorization system} in $\cA$ is given by two classes of arrows
$$\cE \subseteq \arr(\cA) \,,\; \cM \subseteq \arr(\cA)$$
such that
\begin{enumerate}
\item[1)] Both classes are stable under composition with isomorphisms.
\item[2)] Each arrow of $\cA$ can be factorized as an arrow in $\cE$ followed by an arrow in $\cM.$
\item[3)] (Orthogonality) For each solid commutative diagram with $e \in \cE$ and $m \in \cM$
$$\xymatrix{E \ar[r]^{h} \ar[d]_{e} & M \ar[d]^{m} \\ E_0 \ar[r]_{h_0} \ar@{-->}[ru]^{\lambda} & M_0}$$
there exists a unique arrow $\lambda$ such that $\lambda \cdot e = h$ and $m \cdot \lambda = h_0.$
(Sometimes, the notation $e \perp m$ is used to express this condition.)
\end{enumerate}
}\end{Definition}

\begin{Text}\label{TextFS}{\rm
Observe that Definition \ref{DefFS}, which appears in \cite{GT}, seems weaker than the one in \cite{BO1}, 
where it is required that $\cE$ and $\cM$ contain isomorphisms and are closed under composition. 
If, in Definition \ref{DefFS}, we add that $\cE$ and $\cM$ contain identities, we get precisely the definition of 
orthogonal factorization system given in \cite{CJKP}, and it is known that the definitions in \cite{BO1} 
and \cite{CJKP} are equivalent. The fact that Definition \ref{DefFS} is in fact equivalent to the one given in 
\cite{CJKP} can be checked directly, but can also be seen as an obvious consequence of the correspondence 
between orthogonal factorization systems in $\cA$ and $\cH(\cA)$-torsion theories in $\Arr(\cA),$ see Corollary
 \ref{CorOFSid}.
}\end{Text}

The following lemma (and its dual), inspired by the description of homotopy (co)kernels in $\Arr(\cA)$ 
given in Example \ref{ExArrHKHC}, will allow us to state one of our main results, Proposition \ref{PropOFSHTT}, 
without assuming the existence of pullbacks or pushouts in $\cA.$

\begin{Lemma}\label{LemmaIso} 
Consider an adjunction 
$$\xymatrix{\cA \ar@<-0.5ex>[rr]_{\cU} & & \cB \ar@<-0.5ex>[ll]_{\cC} } ,\; \cC \dashv \cU$$
with unit $\gamma_B \colon B \to \cU\cC B$ and counit $\delta_A \colon \cC\cU A \to A$ and assume that $\cU$ is
full and faithful. Put on $\cB$ the structure of nullhomotopies $\Theta_{\gamma}$ as in \ref{TextRad}.2. If
$$\xymatrix{X \ar[rr]_{g} \ar@/^1.8pc/@{-->}[rrrr]^{}_{\theta_g \; \Downarrow} & & Y \ar[rr]_-{q_g} & & \cQ(g) }$$
is a $\Theta_{\gamma}$-cokernel, then the arrow $\cC(q_g) \colon \cC Y \to \cC\cQ(g)$ is an isomorphism. 
\end{Lemma}

\begin{proof}
Since, by naturality of $\gamma,$ we have $\gamma_Y \cdot g = \cU\cC(g) \cdot \gamma_X,$ we can 
look at the arrow $\cU\cC(g) \colon \cU\cC X \to \cU\cC Y$ as a nullhomotopy on $\gamma_Y \cdot g.$ 
By the universal property of the $\Theta_{\gamma}$-cokernel of $g,$ we get a unique arrow $h \colon \cQ(g) \to \cU\cC Y$ 
such that $h \cdot q_g = \gamma_Y$ and $h \circ \theta_g = \cU\cC(g),$ that is, $h \cdot \theta_g = \cU\cC(g).$
From the first condition on $h,$ and using one of the triangular identities, it follows that $\cC(q_g)$ is a split mono:
$$\delta_{\cC Y} \cdot \cC(h) \cdot \cC(q_g) = \delta_{\cC Y} \cdot \cC(\gamma_Y) = \id_{\cC Y}$$
Consider now the diagram
$$\xymatrix{X \ar[rr]_{g} \ar@/^1.8pc/@{-->}[rrrr]^{}_{\theta_g \; \Downarrow} & & Y \ar[rr]_-{q_g} 
& & \cQ(g) \ar[rr]_-{h} \ar@/^1.8pc/[rrrr]^-{\gamma_{\cQ(g)}} & & \cU\cC Y \ar[rr]_-{\cU\cC(q_g)} & & \cU\cC\cQ(g) }$$ 
and observe that
$$\cU\cC(q_g) \cdot h \cdot q_g = \cU\cC(q_g) \cdot \gamma_Y = \gamma_{\cQ(g)} \cdot q_g$$
$$\cU\cC(q_g) \cdot h \cdot \theta_g = \cU\cC(q_g) \cdot \cU\cC(g) = \cU\cC(\theta_g) \cdot \cU\cC(\gamma_X) = 
\cU\cC(\theta_g) \cdot \gamma_{\cU\cC X} = \gamma_{\cQ(g)} \cdot \theta_g$$
(Note that $\cU\cC(\gamma_X) = \gamma_{\cU\cC X}$ because, by the triangular identities, both are left-inverse
of $\cU(\delta_{\cC X})$ which is an isomorphism since $\cU$ is full and faithful.) Therefore, we can use Remark 
\ref{RemCanc} to deduce that $\cU\cC(q_g) \cdot h = \gamma_{\cQ(g)}.$ If we apply now the functor $\cC,$ we obtain that
$\cC\cU\cC(q_g) \cdot \cC(h) = \cC(\gamma_{\cQ(g)}).$ Since $\cC(\gamma_{\cQ(g)})$ is an isomorphism (it is left-inverse
of the isomorphism $\delta_{\cC\cQ(g)}$), this equation implies that $\cC\cU\cC(q_g)$ is a split epi. But it is also a split mono
(because  $\cC(q_g)$ is a split mono) and then it is an isomorphism. Finally, in the commutative diagram
$$\xymatrix{\cC\cU\cC Y \ar[rr]^-{\cC\cU\cC(q_g)} \ar[d]_{\delta_{\cC Y}} & & \cC\cU\cC\cQ(g) \ar[d]^{\delta_{\cC\cQ(g)}} \\ 
\cC Y \ar[rr]_-{\cC(q_g)} & & \cC\cQ(g) }$$
$\cC\cU\cC(q_g), \delta_Y$ and $\delta_{\cQ(g)}$ are isomorphisms, so that $\cC(q_g)$ also is an isomorphism.
\end{proof}

\begin{Proposition}\label{PropOFSHTT}
Let $\cA$ be a category. Consider the category $\Arr(\cA)$ equipped with the structure of nullhomotopies 
$\cH(\cA)$ induced by the adjunctions $\cC \dashv \cU \dashv \cD$ as in Example \ref{ExArr}. 
Orthogonal factorization systems in $\cA$  correspond to $\cH(\cA)$-torsion theories in $\Arr(\cA).$
\end{Proposition}

\begin{proof}
Let $(\cE,\cM)$ be an orthogonal factorization system in $\cA.$ We put 
\begin{enumerate}
\item[-] $\cT=$ the full subcategory of $\Arr(\cA)$ spanned by the objects $(E,e,E_0)$ with $e \in \cE,$
\item[-] $\cF=$ the full subcategory of $\Arr(\cA)$ spanned by the objects $(M,m,M_0)$ with $m \in \cM.$
\end{enumerate}
Conversely, let $(\cT,\cF)$ be a $\cH(\cA)$-torsion theory in $\Arr(\cA).$ We put
\begin{enumerate}
\item[-] $\cE=$ the class of arrows in $\cA$ of the form $t \colon T \to T_0,$ for $(T,t,T_0)$ in $\cT,$
\item[-] $\cM=$ the class of arrows in $\cA$ of the form $f \colon F \to F_0,$ for $(F,f,F_0)$ in $\cF.$
\end{enumerate}
We have to prove that the three conditions on $\cE$ and $\cM$ of Definition \ref{DefFS} correspond to their homologous conditions
on $\cT$ and $\cF$ of Definition \ref{DefHTT}. \\
1) Obviously, $\cT$ is replete iff $\cE$ is closed under composition with isomorphisms. The same holds for $\cF$ and $\cM.$ \\
2) Let $(X,x,X_0)$ be an object in $\Arr(\cA).$ Using its $(\cE,\cM)$-factorization in $\cA$
$$\xymatrix{X \ar[rr]^{x} \ar[rd]_{e_x} & & X_0 \\ & I(x) \ar[ru]_{m_x} }$$
we can construct the following diagram in $\Arr(\cA) \colon$
$$\xymatrix{X \ar[r]^{\id} \ar[d]_{e_x} & X \ar[r]^{e_x} \ar[d]_<<<{x} & I(x) \ar[d]^{m_x} \\
I(x) \ar[rru]^<<<<<{\id} \ar[r]_{m_x} & X_0 \ar[r]_{\id} & X_0 }$$
As already observed in Example \ref{ExHExt}, if we compare the previous diagram with the description of $\cH(\cA)$-kernels 
in $\Arr(\cA)$ given in Example \ref{ExArrHKHC} (and with the dual description of $\cH(\cA)$-cokernels), we immediately see that 
the triple $((X,e_x,I(x)), (\id_X,m_x), \id_{I(x)})$ 
is a $\cH(\cA)$-kernel of $(e_x,\id_{X_0})$ and that the triple $((I(x),m_x,X_0), (e_x,\id_{X_0}), \id_{I(x)})$ is a $\cH(\cA)$-cokernel 
of $(\id_X,m_x).$ Moreover, the object $(X,e_x,I(x))$ is in $\cT$ because $e_x \in \cE$ and the object $(I(x),m_x,X_0)$ is in 
$\cF$ because $m_x \in \cM.$ We can conclude that the above diagram is a $\cH(\cA)$-exact $(\cT,\cF)$-presentation in 
$\Arr(\cA)$ of the object $(X,x,X_0).$ \\
Conversely, we are going to construct the $(\cE,\cM)$-factorization of an arrow in $\cA.$ Fix an arrow $x \colon X \to X_0$ 
in $\cA,$ look at it as an object in $\Arr(\cA)$ and consider the $\cH(\cA)$-exact $(\cT,\cF)$-presentation 
$$\xymatrix{T(x) \ar[r]^-{t} \ar[d]_{t_x} & X \ar[r]^-{f} \ar[d]_<<<{x} & F(x) \ar[d]^{f_x} \\
T_0(x) \ar[rru]^<<<<<{\xi_x} \ar[r]_-{t_0} & X_0 \ar[r]_-{f_0} & F_0(x) }$$
Following Lemma \ref{LemmaIso}, we know that $f_0 \colon X_0 \to F_0(x)$ 
is an isomorphism. Therefore, using the description established in Example \ref{ExArrHKHC}, the $\cH(\cA)$-kernel of $(f,f_0)$ is 
given by the following diagram 
$$\xymatrix{X \ar[rr]^-{\id} \ar[d]_{f} & & X \ar[r]^-{f} \ar[d]_<<<{x} & F(x) \ar[d]^{f_x} \\
F(x) \ar[rrru]^<<<<<<<{\id} \ar[rr]_-{f_0^{-1} \cdot f_x} & & X_0 \ar[r]_-{f_0} & F_0(x) }$$
Since the $\cH(\cA)$-kernel of $(f,f_0)$ is essentially unique, there exists a unique isomorphism 
$(i,i_0) \colon (T(x),t_x,T_0(x)) \to (X,f,F(x))$ in $\Arr(\cA)$ such that
$$(\id_X,f_0^{-1} \cdot f_x) \cdot (i,i_0) = (t,t_0) \;\mbox{ and }\; \id_{F(x)} \circ (i,i_0) = \xi_x$$
The first component of the first condition gives $\id_X \cdot i = t,$ so that $t$ is an isomorphism. The second condition is an 
equation between nullhomotopies which can be rewritten as $\id_{F(x)} \cdot i_0 = \xi_x,$ so that $\xi_x$ also is an isomorphism.
Finally, the $(\cE,\cM)$-factorization of $x \colon X \to X_0$ is depicted by the following commutative diagram, where $t_x \in \cM$ 
and $f_x \in \cE,$
$$\xymatrix{X \ar[rr]^{x} \ar[d]_{t^{-1}}^{\simeq} & & X_0 \\
T(x) \ar[d]_{t_x} & & F_0(x) \ar[u]_{f_0^{-1}}^{\simeq} \\
T_0(x) \ar[rr]_{\xi_x}^{\simeq} & & F(x) \ar[u]_{f_x} }$$
Indeed, $f_0^{-1} \cdot f_x \cdot \xi_x \cdot t_x \cdot t^{-1} = f_0^{-1} \cdot f_x \cdot f \cdot t \cdot t^{-1} = f_0^{-1} \cdot f_x \cdot f
= f_0^{-1} \cdot f_0 \cdot x = x.$ \\
3) Since nullhomotopies in $\cH(\cA)$ are diagonals, the fact that $\cH(\cA)(h,h_0)$ is reduced to a singleton if the 
domain of $(h,h_0)$ is in $\cT$ and the codomain is in  $\cF,$ is precisely the orthogonality condition between arrows in $\cE$ 
and arrows in $\cM.$
\end{proof}

\begin{Corollary}\label{CorOFSid}
Let $(\cE,\cM)$ be an orthogonal factorization system in a category $\cA.$ Then $\cE$ and $\cM$ contain the identity arrows.
\end{Corollary}

\begin{proof}
Via Proposition \ref{PropOFSHTT}, we have to prove that the identity arrows of $\cA,$ seen as objects of $\Arr(\cA),$ are in $\cT \cap \cF,$
where $(\cT,\cF)$ is the $\cH(\cA)$-torsion theory corresponding to $(\cE,\cM).$ By Corollary \ref{CorCarTrivObj}, this amounts to showing
that $\cH(\cA)(\id_{(X,\id_X,X)}) = \{ \ast \},$ which is obvious.
\end{proof}

\begin{Text}\label{TrextFurhVar}{\rm
There are weaker notions than the one of orthogonal factorization system which are relevant for example in the context of abstract homotopy theory. 
Here we consider the one obtained simply dropping the uniqueness in the condition of orthogonality. Some other intermediate notions, especially 
functorial factorizations and algebraic functorial factorizations as in \cite{RT2,GT,RT7,GJ} and the corresponding notions of torsion operators, will 
be revisited from the point of view of homotopy torsion theories in a further work.
}\end{Text} 

\begin{Definition}\label{DefWOFSWHTT}{\rm
\begin{enumerate}
\item[] 
\item Let $(\cB,\Theta)$ be a category with nullhomotopies. A {\it weak $\Theta$-torsion theory} $(\cT,\cF)$ is defined in the 
same way as a $\Theta$-torsion theory (see Definition \ref{DefHTT}) but asking that, for every arrow $h \colon T \to F,$ with 
$T \in \cT$ and $F \in \cF,$ the set $\Theta(h)$ is non-empty (instead of being a singleton, as in \ref{DefHTT}.3).
\item Let $\cA$ be a category. A {\it weakly orthogonal factorization system} is defined in the same way as an orthogonal factorization 
system (see Definition \ref{DefFS}) but without the uniqueness of the diagonal in the orthogonality condition \ref{DefFS}.3.
\end{enumerate}
}\end{Definition} 

\begin{Proposition}\label{PropWOFSWHTT}
Let $\cA$ be a category. Consider the category $\Arr(\cA)$ equipped with the structure of nullhomotopies $\cH(\cA).$
Weakly orthogonal factorization systems in $\cA$  correspond to weak $\cH(\cA)$-torsion theories in $\Arr(\cA).$
\end{Proposition}

\begin{proof}
This follows from a simple inspection of the proof of Proposition \ref{PropOFSHTT}.
\end{proof}

\begin{Remark}\label{RemInvConstr}{\rm
In the proof of Proposition \ref{PropOFSHTT}, we construct a $\cH(\cA)$-exact presentation 
starting from a factorization and vice-versa, as summarized hereunder:
$$\xymatrix{X \ar[rr]^{x} \ar[rd]_{e_x} & & X_0 \\ & I(x) \ar[ru]_{m_x} }
\;\;\; \Rightarrow \;\;\;
\xymatrix{X \ar[r]^{\id} \ar[d]_{e_x} & X \ar[r]^{e_x} \ar[d]_<<<{x} & I(x) \ar[d]^{m_x} \\
I(x) \ar[rru]^<<<<<{\id} \ar[r]_{m_x} & X_0 \ar[r]_{\id} & X_0 }$$
$$\xymatrix{T(x) \ar[r]^-{t} \ar[d]_{t_x} & X \ar[r]^-{f} \ar[d]_<<<{x} & F(x) \ar[d]^{f_x} \\
T_0(x) \ar[rru]^<<<<<{\xi_x} \ar[r]_-{t_0} & X_0 \ar[r]_-{f_0} & F_0(x) }
\;\;\; \Rightarrow \;\;\;
\xymatrix{X \ar[rr]^{x} \ar[d]_{t^{-1}}^{\simeq} & & X_0 \\
T(x) \ar[d]_{t_x} & & F_0(x) \ar[u]_{f_0^{-1}}^{\simeq} \\
T_0(x) \ar[rr]_{\xi_x}^{\simeq} & & F(x) \ar[u]_{f_x} }$$
Let us point out that in the weak case, despite the fact that the factorization of an arrow and the exact presentation 
of an object are no longer essentially unique, the above constructions still realize an essential bijection between the 
factorizations of an arrow $x \colon X \to X_0$ and the exact presentations of the object $(X,x,X_0).$ In fact:
\begin{enumerate}
\item[-] If we start with a factorization, construct an exact presentation and go back to factorizations, we precisely get 
the factorization we started with.
\item[-] If we start with an exact presentation, we construct a factorization and go back to exact presentations, we get
a new exact presentation which is isomorphic to the one we started with. The new exact presentation is
$$\xymatrix{X \ar[rr]^-{\id} \ar[d]_{\xi_x \cdot t_x \cdot t^{-1}} & & X \ar[rr]^-{\xi_x \cdot t_x \cdot t^{-1}} \ar[d]_<<<{x} 
& & F(x) \ar[d]^{f_0^{-1} \cdot f_x} \\
F(x) \ar[rrrru]^<<<<<<<<<<<{\id} \ar[rr]_-{f_0^{-1} \cdot f_x} & & X_0 \ar[rr]_-{\id} & & X_0 }$$
and the isomorphism between the old presentation and the new one is realized by
$$\xymatrix{X \ar[r]^{t^{-1}} \ar[d]_{\xi_x \cdot t_x \cdot t^{-1}} & T(x) \ar[d]^{t_x} \\ F(x) \ar[r]_{\xi_x^{-1}} & T_0(x) }
\;\;\;\;\;\;\;\;\;\;
\xymatrix{F(x) \ar[r]^{\id} \ar[d]_{f_0^{-1} \cdot f_x} & F(x) \ar[d]^{f_x} \\ X_0 \ar[r]_{f_0} & F_0(x) }$$
\end{enumerate}
}\end{Remark} 

\section{The quasi-proper case}\label{SecQP}

As a step towards the case of proper orthogonal factorization systems, which will be treated in Section \ref{SecProper}, 
we find useful to consider the intermediate notion of quasi-proper orthogonal factorization systems. 

\begin{Definition}\label{DefVarFS}{\rm
Let $\cA$ be a category.
\begin{enumerate}
\item A (weakly) orthogonal factorization system $(\cE,\cM)$ is {\it quasi-proper} if:
\begin{enumerate}
\item for every arrow $e \colon E \to E_0$ in $\cE,$ the unit 
$$\gamma_{(E,e,E_0)} = (e,\id_{E_0}) \colon (E,e,E_0) \to (E_0,\id_{E_0},E_0)$$ 
of the adjunction $\cC \dashv \cU$ is an epimorphism in $\Arr(\cA),$
\item for every arrow $m \colon M \to M_0$ in $\cM,$ the counit
$$\beta_{(M,m,M_0)} = (\id_M,m) \colon (M,\id_M,M) \to (M,m,M_0)$$ 
of the adjunction $\cU \dashv \cD$ is a monomorphism in $\Arr(\cA).$ 
\end{enumerate}
\item An orthogonal factorization system $(\cE,\cM)$ is {\it proper} if:
\begin{enumerate}
\item every arrow in $\cE$ is an epimorphism,
\item every arrow in $\cM$ is a monomorphism.
\end{enumerate}
\end{enumerate}
}\end{Definition} 

\begin{Text}\label{TextCommFS}{\rm 
Observe that the notion of proper weakly orthogonal factorization system is irrelevant because, if the arrows in $\cE$ 
are epimorphisms or if the arrows in $\cM$ are monomorphisms, then weakly orthogonal implies orthogonal. 
}\end{Text} 

\begin{Text}\label{TextTermInit}{\rm
In the next lemma, we will use the following simple facts:
\begin{enumerate}
\item If a category $\cA$ has a terminal object $\ast,$ then the functor $\cD \colon \Arr(\cA) \to \cA$ 
has a right adjoint defined by $\Lambda(X) = (X, ! \colon X \to \ast, \ast).$ 
\item If a category $\cA$ has an initial object $\emptyset,$ then the functor $\cC \colon \Arr(\cA) \to \cA$
as a left adjoint defined by $\Gamma(X)=(\emptyset, ! \colon \emptyset \to X, X).$
\end{enumerate}
}\end{Text} 

\begin{Lemma}\label{LemmaTermInit}
Let $\cA$ be a category.
\begin{enumerate}
\item Every proper orthogonal factorization system in $\cA$ is quasi-proper.
\item If the category $\cA$ has a terminal object or an initial object, then every quasi-proper weakly 
orthogonal factorization system is orthogonal.
\item If the category $\cA$ has a terminal object and an initial object, then every quasi-proper weakly 
orthogonal factorization system is proper and orthogonal.
\end{enumerate}
\end{Lemma} 

\begin{proof}
1. If $e \colon E \to E_0$ is an epimorphism in $\cA,$ then $(e,\id_{E_0})$ is an epimorphism in $\Arr(\cA).$
If $m \colon M \to M_0$ is a monomorphism in $\cA,$ then $(\id_M,m)$ is a monomorphism in $\Arr(\cA).$ \\
2 and 3. By assumption, $(e,\id_{E_0})$ is an epimorphism for every $e \colon E \to E_0$ in $\cE.$ If $\cA$ 
has a terminal object, then $\cD \colon \Arr(\cA) \to \cA$ has a right adjoint (see \ref{TextTermInit}) and then 
it preserves epimorphisms. Therefore, $e$ is an epimorphism, which implies that the diagonal in the 
orthogonality condition of Definition \ref{DefFS} is necessarily unique. The argument if $\cA$ has an initial 
object is dual.
\end{proof} 

We have formulated the definition of quasi-proper factorization system in such a way to make easy the comparison with proper 
factorization systems. Here we give an equivalent definition which has the advantage to be transposable to homotopy torsion theories. 

\begin{Lemma}\label{LemmaQPFSbis}
Let $(\cE,\cM)$ be a weakly orthogonal factorization system in a category $\cA.$ For every arrow 
$x \colon X \to X_0$ and for any of its $(\cE,\cM)$-factorization $x = m_x \cdot e_x,$ consider the diagram
$$\xymatrix{X \ar[r]^{\id} \ar[d]_{e_x} & X \ar[r]^{e_x} \ar[d]_{x} & I(x) \ar[d]^{m_x} \\
I(x) \ar[r]_{m_x} & X_0 \ar[r]_{\id} & X_0 }$$
Then $(\cE,\cM)$ is quasi-proper if and only if, in the above diagram, the square on the left is a 
monomorphism in $\Arr(\cA)$ and the square on the right is an epimorphism in $\Arr(\cA).$
\end{Lemma}

\begin{proof}
Assume that the above condition is satisfied. If you take $x$ in $\cE,$ then the square on the right is $\gamma_{(X,x,X_0)};$ 
if you take $x$ in $\cM,$ then the square on the left is $\beta_{(X,x,X_0)}.$ \\
 Conversely, if $(\cE,\cM)$ is quasi-proper, use that the diagram 
$$\xymatrix{(X,e_x,I(x)) \ar[rrrr]^-{\gamma_{(X,e_x,I(x))}=(e_x,\id_{I(x)})} \ar[d]_{(\id_X,m_x)} & & & & 
(I(x),\id_{I(x)},I(x)) \ar[d]^{\beta_{(I(x),m_x,X_0)}=(\id_{I(x)},m_x)} \\
(X,x,X_0) \ar[rrrr]_-{(e_x,\id_{X_0})} & & & & (I(x),m_x,X_0) }$$
is a pullback and a pushout  in $\Arr(\cA).$ 
\end{proof} 

We now move on to homotopy torsion theories.

\begin{Definition}\label{DefVarHTT}{\rm
Let $(\cB,\Theta)$ be a category with nullhomotopies.
A (weak) $\Theta$-torsion theory $(\cT,\cF)$ is {\it quasi-proper} if, in each $\Theta$-exact $(\cT,\cF)$-presentation 
of an object $X$ 
$$\xymatrix{T(X) \ar[rr]_{t_X} \ar@/^1.8pc/@{-->}[rrrr]^{}_{\xi_X \; \Downarrow} & & X \ar[rr]_-{f_X} & & F(X) }$$
the arrow $t_X$ is a monomorphism and the arrow $f_X$ is an epimorphism.
}\end{Definition} 

\begin{Text}\label{TextNotUnique}{\rm 
Some comments on the previous definition.
\begin{enumerate}
\item As already observed, for a weak homotopy torsion theory, the exact presentation of an object is no longer essentially unique. 
This is the reason way, in Definition \ref{DefVarHTT}, we require the condition for any exact presentation of an object. If two exact 
presentations of the same object are isomorphic and one of them satisfies the condition of quasi-properness, the other one also 
satisfies the same condition. 
\item If the structure of nullhomotopies $\Theta$ is discrete, then any $\Theta$-torsion theory is quasi-proper. Indeed, we know from 
Lemma \ref{LemmaZHKer} that $\Theta$-kernels and $\Theta$-cokernels coincide with $\cZ_1(\Theta)$-kernels and 
$\cZ_1(\Theta)$-cokernels which are, respectively, monomorphisms and epimorphisms, as proved in \cite{FF} (see also Remark \ref{RemCanc}.2).
\end{enumerate} 
}\end{Text} 

Now we can adapt Proposition \ref{PropOFSHTT} to the quasi-proper case.

\begin{Proposition}\label{PropPropWOFSHTT}
Let $\cA$ be a category. Consider the category $\Arr(\cA)$ equipped with the structure of nullhomotopies $\cH(\cA).$
Quasi-proper weakly orthogonal factorization systems in $\cA$ correspond to quasi-proper weak  
$\cH(\cA)$-torsion theories in $\Arr(\cA).$
\end{Proposition}

\begin{proof} 
The additional point to prove here is that, for a fixed arrow $x \colon X \to X_0$ in $\cA,$ all the 
$(\cE,\cM)$-factorizations of $x$ satisfy the quasi-properness condition of Definition \ref{DefVarFS}
if and only if all the $\cH(\cA)$-exact $(\cT,\cF)$-presentations of the object $(X,x,X_0)$ satisfy the 
quasi-properness condition of Definition \ref{DefVarHTT}. Everything follows easily from
Lemma \ref{LemmaQPFSbis}, the first point in \ref{TextNotUnique} and Remark \ref{RemInvConstr}.
\end{proof} 

\begin{Corollary}\label{CorPropOFSHTT}
Let $\cA$ be a category. Consider the category $\Arr(\cA)$ equipped with the structure of nullhomotopies $\cH(\cA).$ 
Quasi-proper orthogonal factorization systems in $\cA$ correspond to quasi-proper $\cH(\cA)$-torsion theories in $\Arr(\cA).$
\end{Corollary}

\begin{proof}
This follows combining Proposition \ref{PropOFSHTT} and Proposition \ref{PropPropWOFSHTT}.
\end{proof} 

\section{The proper case}\label{SecProper}

The final output of this section (Corollary \ref{CorPropOFSHTTDis}) will be that, under mild assumptions on the category $\cA,$ 
proper orthogonal factorization systems in $\cA$ correspond to homotopy torsion theories in $\Arr(\cA)$ with respect to the discrete 
structure $\cZ_1(\cA).$ First, we study how homotopy torsion theories react passing from a structure of nullhomotopies to the 
associated discrete structure.  We start with a simple general fact.

\begin{Proposition}\label{PropPropHTTChngNull}
Let $(\cB,\Theta)$ be a category with nullhomotopies.  
If $(\cT,\cF)$ is a quasi-proper weak $\Theta$-torsion theory, then it is also a $\cZ_1(\Theta)$-torsion theory.
\end{Proposition}

\begin{proof}
Consider an object $X$ in $\cB$ and one of its $\Theta$-exact $(\cT,\cF)$-presentations
$$\xymatrix{T(X) \ar[rr]_{t_X} \ar@/^1.8pc/@{-->}[rrrr]^{}_{\xi_X \; \Downarrow} & & X \ar[rr]_-{f_X} & & F(X) }$$
We are going to prove that, forgetting the nullhomotopy $\xi_X,$ the same diagram provides a presentation which is $\cZ_1(\Theta)$-exact. 
We check the universal property of the $\cZ_1(\Theta)$-kernel, the one for the $\cZ_1(\Theta)$-cokernel is dual. 
Let $f \colon W \to X$ be an arrow such that $f_X \cdot f \in \cZ_1(\Theta).$ This means that there exists a nullhomotopy 
$\varphi \in \Theta(f_X \cdot f).$ Therefore, there exists a unique arrow $f' \colon W \to T(X)$ such that $t_X \cdot f' = f$ and $\xi_X \circ f' = \varphi.$ 
If $f'' \colon W \to T(X)$ is such that $t_X \cdot f'' = f,$ then $f' = f''$ because, by assumption, $t_X$ is a monomorphism.
\end{proof}

In order to invert Proposition \ref{PropPropHTTChngNull} when $\cB = \Arr(\cA),$ we need the discrete version of Lemma \ref{LemmaIso}.

\begin{Lemma}\label{LemmaIsoDisc}
Consider an adjunction 
$$\xymatrix{\cA \ar@<-0.5ex>[rr]_{\cU} & & \cB \ar@<-0.5ex>[ll]_{\cC} } ,\; \cC \dashv \cU$$
with unit $\gamma_B \colon B \to \cU\cC B$ and counit $\delta_A \colon \cC\cU A \to A$ and assume that $\cU$ is
full and faithful. Consider the structure of nullhomotopies $\Theta_{\gamma}$ on $\cB$ and the ideal $\cZ_1(\Theta_{\gamma}).$ If
$$\xymatrix{X \ar[rr]^{g} & & Y \ar[rr]^-{q_g} & & \cQ(g) }$$
is a $\cZ_1(\Theta_{\gamma})$-cokernel, then the arrow $\cC(q_g) \colon \cC Y \to \cC\cQ(g)$ is an isomorphism. 
\end{Lemma}

\begin{proof}
We follow the same lines of the proof of Lemma \ref{LemmaIso}. Since, by naturality of $\gamma,$ 
we have $\gamma_Y \cdot g = \cU\cC(g) \cdot \gamma_X,$ the arrow $\gamma_Y \cdot g \colon X \to Y \to \cU\cC Y$ 
is in $\cZ_1(\Theta_{\gamma}).$ By the universal property of the $\cZ_1(\Theta_{\gamma})$-cokernel of $g,$ we get a unique arrow 
$h \colon \cQ(g) \to \cU\cC Y$ such that $h \cdot q_g = \gamma_Y.$ Using one of the triangular identities, it follows 
from the previous equation that $\cC(q_g)$ is a split mono as in the proof of \ref{LemmaIso}. But $\cC(q_g)$ is also 
an epimorphism because $q_g$ is an epimorphism (see \ref{TextNotUnique}.2) and $\cC$ is a left adjoint. 
\end{proof}

\begin{Proposition}\label{PropQPWHTTchang}
Let $\cA$ be a category. Consider in $\Arr(\cA)$ the structure of nullhomotopies $\cH(\cA)$ and the associated discrete structure 
$\cZ_1(\cA).$ If $(\cT,\cF)$ is a $\cZ_1(\cA)$-torsion theory, then it is also a quasi-proper weak $\cH(\cA)$-torsion theory.
\end{Proposition}

\begin{proof}
Fix an arrow $x \colon X \to X_0$ in $\cA$ and consider a $\cZ_1(\cA)$-exact $(\cT,\cF)$-presentation of the object $(X,x,X_0) \colon$
$$\xymatrix{T(x) \ar[r]^-{t} \ar[d]_{t_x} & X \ar[r]^-{f} \ar[d]_{x} & F(x) \ar[d]^{f_x} \\
T_0(x) \ar[r]_-{t_0} & X_0 \ar[r]_-{f_0} & F_0(x) }$$
By Lemma \ref{LemmaIsoDisc} and its dual, $t$ and $f_0$ are isomorphisms. Moreover, $(t,t_0)$ is a monomorphism (because it is a 
$\cZ_1(\cA)$-kernel) and $(f,f_0)$ is an epimorphism (because it is a $\cZ_1(\cA)$-cokernel). Since $(f,f_0) \cdot (t,t_0) \in \cZ_1(\cA),$
there exists an arrow $\xi_x \colon T_0(x) \to F(x)$ such that $f \cdot t = \xi_x \cdot t_x$ and $f_0 \cdot t_0 = f_x \cdot \xi_x.$
We are going to prove that the above diagram, completed with the arrow $\xi_x,$ is a $\cH(\cA)$-exact $(\cT,\cF)$-presentation
of $(X,x,X_0).$ We follow as far as possible the proof of Proposition \ref{PropOFSHTT}. \\
First, observe that, since $f_0$ is an isomorphism, we can use \ref{ExArrHKHC} and the $\cH(\cA)$-kernel of 
$(f,f_0) \colon (X,x,X_0) \to (F(x),f_x,F_0(x))$ is
$$\xymatrix{X \ar[rr]^-{\id} \ar[d]_{f} & & X \ar[r]^-{f} \ar[d]_<<<{x} & F(x) \ar[d]^{f_x} \\
F(x) \ar[rrru]^<<<<<<<{\id} \ar[rr]_-{f_0^{-1} \cdot f_x} & & X_0 \ar[r]_-{f_0} & F_0(x) }$$
By its universal property, we get a unique arrow $(i,i_0) \colon (T(x),t_x,T_0(x)) \to (X,f,F(x))$ such that
$(\id_X,f_0^{-1} \cdot f_x) \cdot (i,i_0) = (t,t_0)$ and $\id_{F(x)} \circ (i,i_0) = \xi_x.$ This implies that $i=t,$ so that $i$ is an isomorphism,
and $i_0 = \xi_x.$ Moreover, by the universal property of the $\cZ_1(\cA)$-kernel of $(f,f_0),$ we get a unique arrow 
$(j,j_0) \colon (X,f,F(x)) \to (T(x),t_x,T_0(x))$ such that $(t,t_0) \cdot (j,j_0) = (\id_X,f_0^{-1} \cdot f_x).$ Since $(t,t_0) \cdot (j,j_0) \cdot (i,i_0) = (t,t_0)$
and $(t,t_0)$ is a monomorphism, we have $(j,j_0) \cdot (i,i_0) = (\id_{T(x)},\id_{T_0(x)}).$ This implies that $i_0$ is a split monomorphism,
that is, $\xi_x$ is a split monomorphism. \\
Second, observe that, since $t$ is an isomorphism, we can use the dual of \ref{ExArrHKHC} and the $\cH(\cA)$-cokernel of 
$(t,t_0) \colon (T(x),t_x,T_0(x)) \to (X,x,X_0)$ is
$$\xymatrix{T(x) \ar[r]^{t} \ar[d]_{t_x} & X \ar[rr]^-{t_x \cdot t^{-1}} \ar[d]_<<<<{x} & & T_0(x) \ar[d]^{t_0} \\
T_0(x) \ar[r]_{t_0} \ar[rrru]_>>>>>>>>>>{\id} & X_0 \ar[rr]_-{\id} & & X_0 }$$
By its universal property, we get a unique arrow $(m,m_0) \colon (T_0(x),t_0,X_0) \to (F(x),f_x,F_0(x))$ such that
$(m,m_0) \cdot (t_x \cdot t^{-1} ,\id_{X_0}) = (f,f_0)$ and $(m,m_0) \circ \id_{T_0(x)} = \xi_x.$ This implies $m_0=f_0,$ so that $m_0$ is an 
isomorphism, and $m=\xi_x.$ Moreover, by the universal property of the $\cZ_1(\cA)$-cokernel of $(t,t_0),$ we get a unique arrow 
$(n,n_0) \colon (F(x),f_x,F_0(x)) \to (T_0(x),t_0,X_0)$ such that $(n,n_0) \cdot (f,f_0) = (t_x \cdot t^{-1},\id_{X_0}).$ Since
$(m,m_0) \cdot (n,n_0) \cdot (f,f_0) = (f,f_0)$ and $(f,f_0)$ is an epimorphism, we have $(m,m_0) \cdot (n,n_0) = (\id_{F(x)},\id_{F_0(x)}).$
This implies that $m$ is a split epimorphism, that is, $\xi_x$ is a split epimorphism.\\
We can conclude that $\xi_x,$ being a split monomorphism and a split epimorphism, is an isomorphism. This implies that both $(i,i_0)$ and
$(m,m_0)$ are isomorphisms and, therefore, 
$$\xymatrix{T(x) \ar[r]^-{t} \ar[d]_{t_x} & X \ar[r]^-{f} \ar[d]_<<<{x} & F(x) \ar[d]^{f_x} \\
T_0(x) \ar[rru]^<<<<<{\xi_x} \ar[r]_-{t_0} & X_0 \ar[r]_-{f_0} & F_0(x) }$$
is a $\cH(\cA)$-exact $(\cT,\cF)$-presentation of $(X,x,X_0),$ and it is also quasi-proper because $(t,t_0)$ is a monomorphism and $(f,f_0)$
is an epimorphism.\\
As far as the weak version of condition 3 in Definition \ref{DefHTT} is concerned, observe that, if $(h,h_0) \colon (T,t,T_0) \to (F,f,F_0)$
is an arrow in $\Arr(\cA)$ with $(T,t,T_0) \in \cT$ and $(F,f,F_0) \in \cF,$ then $(h,h_0) \in \cZ_1(\cA),$ so that $\cH(\cA)(h,h_0) \neq \emptyset.$
\end{proof} 

If we assume that the category $\cA$ has an initial object or a terminal object, then Proposition \ref{PropQPWHTTchang} 
can be slightly improved. 

\begin{Proposition}\label{PropQPHTTchang}
Let $\cA$ be a category with an initial or a terminal object. Consider in $\Arr(\cA)$ the structure of nullhomotopies 
$\cH(\cA)$ and the associated discrete structure $\cZ_1(\cA).$ If $(\cT,\cF)$ is a $\cZ_1(\cA)$-torsion theory, 
then it is also a quasi-proper $\cH(\cA)$-torsion theory.
\end{Proposition}

\begin{proof}
We proceed as in the proof of Proposition \ref{PropQPWHTTchang}, so that it remains to check the uniqueness of the
nullhomotopy in condition 3 of Definition \ref{DefHTT}. Assume that the category $\cA$ has a terminal object
(if the category $\cA$ has an initial object, the argument is dual). This implies that the functor $\cD \colon \Arr(\cA) \to \cA$ 
preserves epimorphisms (see \ref{TextTermInit}) so that, in the $\cH(\cA)$-exact $(\cT,\cF)$-presentation 
$$\xymatrix{T(x) \ar[r]^-{t} \ar[d]_{t_x} & X \ar[r]^-{f} \ar[d]_<<<{x} & F(x) \ar[d]^{f_x} \\
T_0(x) \ar[rru]^<<<<<{\xi_x} \ar[r]_-{t_0} & X_0 \ar[r]_-{f_0} & F_0(x) }$$
the arrow $f = \cD(f,f_0)$ is an epimorphism. Now in the commutative square
$$\xymatrix{T(x) \ar[r]^{t} \ar[d]_{t_x} & X \ar[d]^{f} \\ T_0(x) \ar[r]_{\xi_x} & F(x) }$$
$t$ and $\xi_x$ are isomorphisms and $f$ is an epimorphism, so that $t_x$ is an epimorphism. \\
Consider now an object $(X,x,X_0) \in \cT.$ As $\cZ_1(\cA)$-exact $(\cT,\cF)$-presentation we can choose one of the form
$$\xymatrix{X \ar[r]^{\id} \ar[d]_{x} & X \ar[r]^{f} \ar[d]^{x} & F(x) \ar[d]^{f_x} \\
X_0 \ar[r]_{\id} & X_0 \ar[r]_{f_0} & F_0(x) }$$
(see Corollary \ref{CorCarTF}). If we apply the previous argument to this presentation, we can deduce that such an arrow 
$x \colon X \to X_0$ is an epimorphism. \\
Finally, consider an arrow $(h,h_0) \colon (T,t,T_0) \to (F,f,F_0)$ with domain $(T,t,T_0) \in \cT$ and codomain $(F,f,F_0) \in \cF.$ 
Since $(\cT,\cF)$ is a $\cZ_1(\cA)$-torsion theory, we necessarily have that $(h,h_0) \in \cZ_1(\cA).$ This means that there exists 
a nullhomotopy $\lambda \in \cH(\cA)(h,h_0),$ that is, an arrow $\lambda \colon T_0 \to F$ such that the diagram
$$\xymatrix{T \ar[rr]^{h} \ar[d]_{t} & & F \ar[d]^{f} \\ T_0 \ar[rru]^{\lambda} \ar[rr]_{h_0} & & F_0 }$$
commutes in each part. Since, by the previous argument, $t$ is an epimorphism, such an arrow $\lambda$ is unique. 
\end{proof} 

\begin{Corollary}\label{CorPropOFSHTTDis}
Let $\cA$ be a category with an initial object and a terminal object. Consider the category $\Arr(\cA)$ equipped 
with the discrete structure of nullhomotopies $\cZ_1(\cA).$ Proper orthogonal factorization systems in $\cA$ 
correspond to $\cZ_1(\cA)$-torsion theories in $\Arr(\cA).$
\end{Corollary}

\begin{proof}
By Propositions \ref{PropPropHTTChngNull} and \ref{PropQPWHTTchang}, $\cZ_1(\cA)$-torsion theories  correspond 
to quasi-proper weak $\cH(\cA)$-torsion theories which, by \ref{PropWOFSWHTT}, correspond to quasi-proper weakly 
orthogonal factorization systems. Finally, if $\cA$ has an initial object and a terminal object, quasi-proper weakly 
orthogonal factorization systems coincide with proper orthogonal factorization systems by Lemma \ref{LemmaTermInit}.
\end{proof} 

\section{Some remarks on the pointed case}\label{SecPointed}

In this section, we will denote the kernel and the cokernel (in the usual sense) of an arrow $g \colon X \to Y$ 
in a category $\cA$ with a zero object as
$$\xymatrix{K(g) \ar[r]^-{k_g} & X \ar[r]^-{g} & Y \ar[r]^-{c_g} & C(g) }$$
This notation is coherent with the one in \ref{DefZK}, because kernels coincide with 
$\cZ_1(0)$-kernels (and cokernels with $\cZ_1(0)$-cokernels), where $\cZ_1(0)$ is the ideal of zero arrows in $\cA,$
as already pointed out in Remark \ref{RemPointed}. 

\begin{Text}\label{TextDLK}{\rm
If we assume that a category $\cA$ has a zero object and kernels, we obtain a new string of adjunctions
$$\xymatrix{ \cA \ar[rr]|-{\Lambda} & & \Arr(\cA) \ar@<-1.5ex>[ll]_-{\cD} \ar@<1.5ex>[ll]^-{\Ker} }
\;\;\;\;\; \cD \dashv \Lambda \dashv \Ker$$
The functor $\cD \colon \Arr(\cA) \to \cA$ is the one of Example \ref{ExArr}. The full and faithful functor $\Lambda \colon \cA \to \Arr(\cA),$
defined by $\Lambda(X) = (X,!,0),$ is the one introduced in \ref{TextTermInit} (we have changed notation writing 0 for the 
zero object, whereas in \ref{TextTermInit} we used $\ast$ for the terminal object). The functor $\Ker \colon \Arr(\cA) \to \cA$ 
is defined on objects by $\Ker(X,x,X_0) = K(x)$ and is extended to arrows in the obvious way. \\
The unit $\gamma_{(X,x,X_0)} \colon (X,x,X_0) \to \Lambda\cD(X,x,X_0)$ of $\cD \dashv \Lambda$ is
$\xymatrix{X \ar[r]^{\id} \ar[d]_x & X \ar[d]^{!} \\ X_0 \ar[r]_{!} & 0}$ \\
The counit $\beta_{(Y,y,Y_0)} \colon \Lambda\Ker(Y,y,Y_0) \to (Y,y,Y_0)$ of $\Lambda \dashv \Ker$ is
$\xymatrix{K(y) \ar[r]^-{k_y} \ar[d]_{!} & Y \ar[d]^y \\ 0 \ar[r]_{!} & Y_0}$ \\
Following the constructions explained in \ref{TextRad} and applying Proposition \ref{PropCUD}, it turns out that the three isomorphic 
structures of nullhomotopies induced on $\Arr(\cA)$ by these adjunctions are discrete (this follows from the fact that the unit 
$\gamma_{(X,x,X_0)}$ is an epimorphism or, equivalently, by the fact that the counit $\beta_{(Y,y,Y_0)}$ is a monomorphism, cf. Remark 
\ref{RemTriv}.2) and the corresponding ideal of arrows is
$$\cZ_1(\Lambda) = \{ (g,g_0) \colon (X,x,X_0) \to (Y,y,Y_0) \mid g_0 = 0 \colon X_0 \to Y_0 \}$$
Following Proposition \ref{PropExistenceHKHC} and Lemma \ref{LemmaUnitHCoker}, we can describe $\cZ_1(\Lambda)$-kernels and, 
if we assume also the existence of cokernels in $\cA,$ $\cZ_1(\Lambda)$-cokernels in $\Arr(\cA) \colon$
$$\xymatrix{K(y \cdot g) \ar[r]^-{k_{y \cdot g}} \ar[d]_{x'} & X \ar[r]^{g} \ar[d]_{x} & Y \ar[d]^{y} \\
K(g_0) \ar[r]_-{k_{g_0}} & X_0 \ar[r]_{g_0} & Y_0 }
\;\;\;\;\;\;\;\;\;\;
\xymatrix{X \ar[r]^-{g} \ar[d]_{x} & Y \ar[r]^-{\id} \ar[d]^{y} & Y \ar[d]^{c_{g_0} \cdot y} \\
X_0 \ar[r]_-{g_0} & Y_0 \ar[r]_-{c_{g_0}} & C(g_0) }$$
where $x'$ is the unique arrow such that the first square on the left commutes.
}\end{Text} 

Recall, from Remark \ref{RemPointed}, that $\cZ_1(0)$-torsion theories are the usual torsion theories in a category with zero object.

\begin{Proposition}\label{PropPointed} 
Let $\cA$ be a category with zero object, kernels and cokernels.
Every $\cZ_1(0)$-torsion theory in $\cA$ induces a $\cZ_1(\Lambda)$-torsion theory in $\Arr(\cA).$
\end{Proposition}

\begin{proof}
Let $(\cT,\cF)$ be a $\cZ_1(0)$-torsion theory in $\cA.$ We put:
\begin{enumerate}
\item[-] $\cT_{\Lambda}=$ the full subcategory of $\Arr(\cA)$ spanned by the objects $(X,x,X_0)$ with $X_0 \in \cT,$
\item[-] $\cF_{\Lambda}=$ the full subcategory of $\Arr(\cA)$ spanned by the objects $(Y,y,Y_0)$ with $Y_0 \in \cF.$
\end{enumerate}
1) $\cT_{\Lambda}$ and $\cF_{\Lambda}$ are replete because $\cT$ and $\cF$ are replete. \\
2) Consider an object $(A,a,A_0)$ in $\Arr(\cA).$ Its $\cZ_1(\Lambda)$-exact $(\cF_{\Lambda},\cT_{\Lambda})$-presentation
is completely determined by the $\cZ_1(0)$-exact $(\cT,\cF)$-presentation of $A_0$
$$\xymatrix{T(A_0) \ar[r]^-{t_{A_0}} & A_0 \ar[r]^-{f_{A_0}} & F(A_0) }$$ 
and by the description of $\cZ_1(\Lambda)$-kernels and $\cZ_1(\Lambda)$-cokernels given in \ref{TextDLK}. The resulting 
diagram in $\Arr(\cA)$ is
$$\xymatrix{K(f_{A_0} \cdot a) \ar[rr]^-{k_{f_{A_0} \cdot a}} \ar[d]_{a'} & & A \ar[rr]^-{\id} \ar[d]^{a} & & A \ar[d]^{f_{A_0} \cdot a} \\
T(A_0) \ar[rr]_-{t_{A_0}} & & A_0 \ar[rr]_-{f_{A_0}} & & F(A_0) }$$
where $a'$ is the unique arrow such that the first square commutes. To check that this is the needed presentation of $(A,a,A_0)$
is easy, keeping in mind that $t_{A_0}$ is the kernel of $f_{A_0}$ and $f_{A_0}$ is the cokernel of $t_{A_0}.$ \\
3) Consider an arrow $(g,g_0) \colon (X,x,X_0) \to (Y,y,Y_0)$ with domain in $\cT_{\Lambda}$ and codomain in $\cF_{\Lambda}.$
The arrow $g_0 \colon X_0 \to Y_0$ has domain in $\cT$ and codomain in $\cF,$ so that $g_0$ is a zero arrow and then $(g,g_0)$
is in $\cZ_1(\Lambda).$
\end{proof}

\begin{Remark}\label{RemPointedInvert}{\rm
Observe that a $\cZ_1(\Lambda)$-torsion theory $(\cT_{\Lambda},\cF_{\Lambda})$ in $\Arr(\cA)$ is induced by a $\cZ_1(0)$-torsion 
theory in $\cA$ as in Proposition \ref{PropPointed} if and only if, given objects $(X,x,X_0) \in \cT_{\Lambda}$ and $(Y,y,Y_0) \in \cF_{\Lambda},$ 
the unique arrow $X_0 \to Y_0$ is the zero arrow. 
}\end{Remark} 

\section{A panoramic view}\label{SecPanoramic}

Sections \ref{SecFSHTT}, \ref{SecQP} and \ref{SecProper} have been devoted to compare various types of factorization systems in a category $\cA$ to various 
types of homotopy torsion theories in $\Arr(\cA).$ If the reader may be confused by all the variants involved, he/she can refer to the following 
panoramic view. 

The acronyms are as follows: Q = quasi, P = proper, W = weak or weakly, O = orthogonal, FS = factorization system, HTT = homotopy torsion theory.
The unlabelled arrows are obvious implications. The name of the other arrows are internal references. 
Recall that items \ref{LemmaTermInit}.2 and \ref{PropQPHTTchang} require that $\cA$ has an initial object {\it or} a terminal object, whereas 
items \ref{LemmaTermInit}.3 and \ref{CorPropOFSHTTDis} require that $\cA$ has an initial object {\it and} a terminal object.
\begin{enumerate}
\item Synopsis for factorization systems:
$$
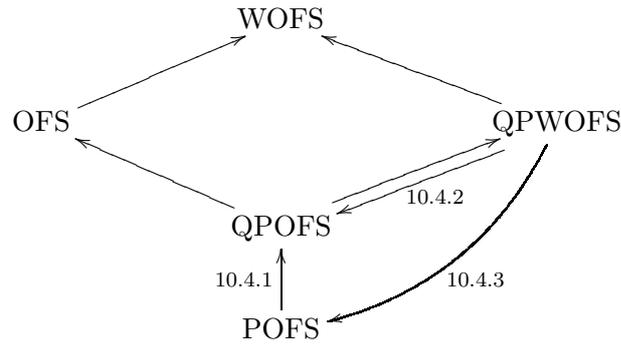
\xymatrix{ && \mbox{WOFS} \\
\mbox{OFS} \ar[urr] &&&& \mbox{QPWOFS} \ar[llu] \ar@<0.5ex>[lld]^{10.4.2} \ar@/^1.5pc/[lldd]^{10.4.3} \\
&& \mbox{QPOFS} \ar[llu] \ar@<0.5ex>[rru] \\
&& \mbox{POFS} \ar[u]^{10.4.1} }$$
\item Synopsis for homotopy torsion theories:
$$
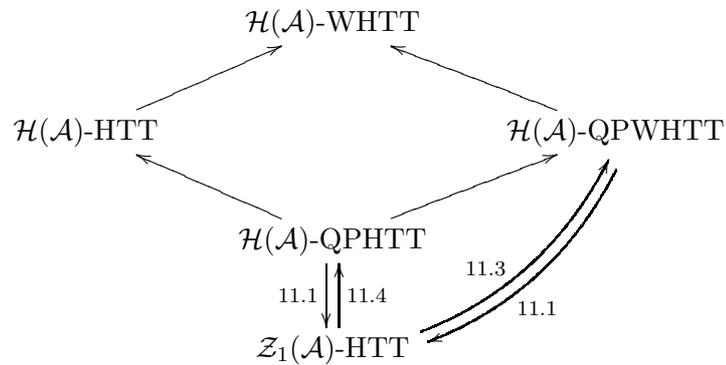
\xymatrix{ & \cH(\cA)\mbox{-WHTT} \\
\cH(\cA)\mbox{-HTT} \ar[ru] & & \cH(\cA)\mbox{-QPWHTT} \ar[lu] \ar@/^1.5pc/@<1.5ex>[ldd]^{11.1} \\
& \cH(\cA)\mbox{-QPHTT} \ar[ru] \ar[lu] \ar@<-0.5ex>[d]_{11.1} \\
& \cZ_1(\cA)\mbox{-HTT} \ar@<-0.5ex>[u]_{11.4} \ar@/_1.5pc/@<-0.5ex>[ruu]^{11.3} }$$
\item Comparison between factorization systems and homotopy torsion theories:
$$
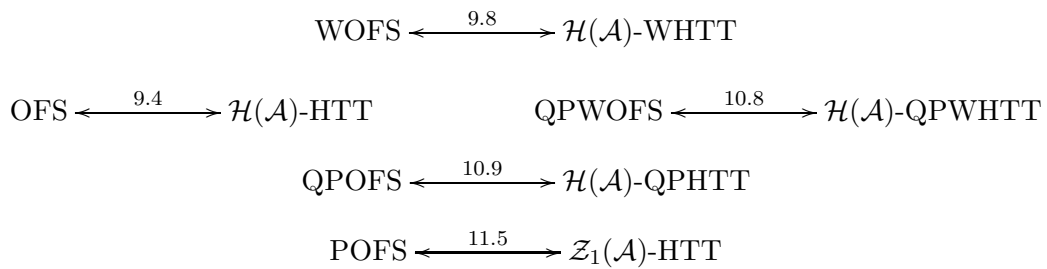
\xymatrix{\mbox{WOFS} \ar@{<->}[rr]^-{9.8} && \cH(\cA)\mbox{-WHTT} }$$
$$\xymatrix{\mbox{OFS} \ar@{<->}[rr]^-{9.4} && \cH(\cA)\mbox{-HTT} && \mbox{QPWOFS} \ar@{<->}[rr]^-{10.8} && \cH(\cA)\mbox{-QPWHTT}  }$$
$$\xymatrix{\mbox{QPOFS} \ar@{<->}[rr]^-{10.9} && \cH(\cA)\mbox{-QPHTT}  }$$
$$\xymatrix{\mbox{POFS} \ar@{<->}[rr]^-{11.5} && \cZ_1(\cA)\mbox{-HTT}  }$$
\end{enumerate}


\begin{thebibliography}{}
\bibliographystyle{alpha}

\bibitem{BO1} {\sc F. Borceux,}  Handbook of categorical algebra, vol. 1.
Cambridge University Press (1994) xvi+345 pp.

\bibitem{BO2} {\sc F. Borceux,}  Handbook of categorical algebra, vol. 2.
Cambridge University Press (1994) xvii+443 pp.

\bibitem{CJKP} {\sc A. Carboni, G. Janelidze, G.M. Kelly, R. Par\`e,} On localization and stabilization for factorization systems, 
{\em Applied Categorical Structures} 5 (1997) 1--58.

\bibitem{CLF} {\sc M.M. Clementino, I. L\`opez Franco,} Lax orthogonal factorisation systems,
{\em Advances in Mathematics} 302 (2016) 458--528. 

\bibitem{Dick} {\sc S.E. Dickson,} A torsion theory for abelian categories, 
{\em Transactions of the American Mathematical Society} 121 (1966) 223--235.

\bibitem{FF} {\sc A. Facchini, C. Finocchiaro,} Pretorsion theories, stable category and preordered sets, 
{\em Annali di Matematica Pura e Applicata} 199 (2020) 1073--1089.

\bibitem{FFG}{\sc A. Facchini, C. Finocchiaro, M. Gran,} Pretorsion theories in general categories,
{Journal of Pure and Applied Algebra} 225 (2021) 21 pp.

\bibitem{GR97} {\sc M. Grandis,} Simplicial homotopical algebra and satellites, 
{\em Applied Categorical Structures} 5 (1997) 75--97.

\bibitem{GJ} {\sc M. Grandis, G. Janelidze,} From torsion theories to closure operators and factorization systems,
{\em Categories and General Algebraic Structures with Application} 12 (2020) 89--121.

\bibitem{GT} {\sc M. Grandis, W. Tholen,} Natural weak factorization systems,
{\em Archivum Mathematicum (Brno)} 42 (2006) 397--408.

\bibitem{JMMVFibr} {\sc P.-A. Jacqmin, S. Mantovani, G. Metere, E.M. Vitale,} On fibrations between internal groupoids and their normalizations,
{\em Applied Categorical Structures} 26 (2018) 1015--1039.

\bibitem{KT} {\sc M. Korostenski, W. Tholen,} Factorization systems as Eilenberg-Moore algebras,
{\em Journal of Pure and Applied Algebra} 85 (1993) 57--72.
 
\bibitem{MM} {\sc M. Mather,} Pull-backs in homotopy theory, 
{\em Canadian Journal of Mathematics} 28 (1976) 225--263.

\bibitem{Sandra} {\sc S. Mantovani,} Torsion theories for crossed modules,
{\em Workshop in Category Theory and Algebraic Topology, Louvain-la-Neuve} (2015) unpublished talk.

\bibitem{MM22} {\sc M. Messora,} Homotopy torsion theories - A homotopical approach to non-pointed torsion theories, 
{\em Master Thesis, Universit\`a degli Studi di Milano} (2022).

\bibitem{RT2} {\sc J. Rosick\'y, W. Tholen,} Lax factorization algebras,
{\em Journal of Pure and Applied Algebra} 175 (2002) 355--382.

\bibitem{RT7} {\sc J. Rosick\'y, W. Tholen,} Factorization, fibration and torsion,
{\em Journal of Homotopy and Related Structures} 2 (2007) 295--314.

\bibitem{SnailEV} {\sc E.M. Vitale,} The snail lemma, 
{\em Theory and Applications of Categories} 31 (2016) 484--501.

\bibitem{EVhcc} {\sc E.M. Vitale,} Completion under strong homotopy kernels
(in preparation).

\bibitem{WCM} {\sc J.H.C. Whitehead,} Combinatorial homotopy II, 
{\em Bulletin of the American Mathematical Society} 55 (1949) 453--496.

\end{thebibliography}
\end{document}